\documentclass[%
  a4paper,
  onecolumn,
  colorlinks,
]{preprint}

\usepackage{amsmath}
\usepackage{amsthm}
\usepackage{amssymb}
\usepackage{dsfont}

\usepackage{graphicx}

\usepackage{tikz}
  \usetikzlibrary{calc}
  \usetikzlibrary{decorations.text}
  \usetikzlibrary{positioning}
  \usetikzlibrary{shapes.geometric}

\usepackage{pgfplotstable}
\usepackage{pgfplots}
  \pgfplotsset{%
    compat          = 1.17,
    colormap name   = viridis,
    axis line style = {black}%
  }
  \pgfkeys{/pgf/number format/.cd,1000 sep={\,}}
  \usepgfplotslibrary{fillbetween}
  \usepgfplotslibrary{external}
  \tikzset{external/system call = {%
    pdflatex \tikzexternalcheckshellescape
      -halt-on-error
      -interaction=batchmode
      -jobname "\image" "\texsource"}}
  \tikzexternaldisable

\usepackage{caption}
\usepackage{subcaption}

\usepackage{booktabs}

\usepackage[linesnumbered, lined, algoruled]{algorithm2e}

\usepackage{cleveref}

\theoremstyle{plain}\newtheorem{theorem}{Theorem}

\newcommand{\R}{\ensuremath{\mathbb{R}}}
\newcommand{\N}{\ensuremath{\mathbb{N}}}

\newcommand{\trans}{\ensuremath{\mkern-1.5mu\mathsf{T}}}
\DeclareMathOperator{\logmean}{logmean}

\newcommand{\Vcal}{\ensuremath{\mathcal{V}}}
\newcommand{\Wcal}{\ensuremath{\mathcal{W}}}
\newcommand{\Xcal}{\ensuremath{\mathcal{X}}}
\newcommand{\Mcal}{\ensuremath{\mathcal{M}}}
\newcommand{\Mcalu}{\ensuremath{\Mcal_{\mathrm{u}}}}

\newcommand{\nh}{\ensuremath{N}}   
\newcommand{\nm}{\ensuremath{M}}   
\newcommand{\nmr}{\ensuremath{m}}  
\newcommand{\np}{\ensuremath{p}}   
\newcommand{\nr}{\ensuremath{n}}   
\newcommand{\nc}{\ensuremath{r}}   

\newcommand{\xs}{\ensuremath{\bar{x}}}
\newcommand{\us}{\ensuremath{\bar{u}}}

\newcommand{\thetar}{\psi}
\newcommand{\Thetar}{\Psi}
\newcommand{\pre}{\ensuremath{\mathrm{pre}}}

\newcommand{\Kt}{\ensuremath{K_{\mathrm{\theta}}}}
\newcommand{\Kr}{\ensuremath{\widetilde{K}}}
\newcommand{\Ktr}{\ensuremath{K_{\mathrm{\thetar}}}}
\newcommand{\Krr}{\ensuremath{\widetilde{K}_{\mathrm{\thetar}}}}

\newcommand{\xr}{\ensuremath{\tilde{x}}}
\newcommand{\ur}{\ensuremath{\tilde{u}}}

\newcommand{\Jx}{\ensuremath{J_{\mathrm{x}}}}
\newcommand{\Ju}{\ensuremath{J_{\mathrm{u}}}}
\newcommand{\Jxr}{\ensuremath{\widetilde{J}_{\mathrm{x}}}}
\newcommand{\Jur}{\ensuremath{\widetilde{J}_{\mathrm{u}}}}

\newcommand{\xc}{\ensuremath{\check{x}}}
\newcommand{\uc}{\ensuremath{\check{u}}}

\newcommand{\tf}{\ensuremath{t_{\mathrm{f}}}}
\newcommand{\ta}{\ensuremath{t_{\mathrm{a}}}}

\newcommand{\lambdacr}{\ensuremath{\tilde{\lambda}^{\mathrm{c}}}}
\newcommand{\lambdacs}{\ensuremath{\lambda_{\mathrm{s}}^{\mathrm{c}}}}
\newcommand{\lambdacu}{\ensuremath{\lambda_{\mathrm{u}}^{\mathrm{c}}}}
\newcommand{\Rn}{\ensuremath{R_{\mathrm{n}}}}

\newcommand{\Ocal}{\ensuremath{\mathcal{O}}}

\newcommand{\FOM}{\mbox{direct~DDPG}}
\newcommand{\Proj}{\mbox{UMPO}}
\newcommand{\ROM}{\mbox{UMPO-MA}}
\newcommand{\ROMProj}{\mbox{MF-UMPO}}
\newcommand{\algname}{\mbox{[MF-]UMPO}}

\let\oldnl\nl
\newcommand{\nonl}{\renewcommand{\nl}{\let\nl\oldnl}}

\definecolor{matlabblue}{HTML}{0072BD}
\definecolor{matlaborange}{HTML}{D95319}
\definecolor{matlabyellow}{HTML}{EDB120}
\definecolor{matlabpurple}{HTML}{7E2F8E}
\definecolor{matlabgreen}{HTML}{77AC30}
\definecolor{matlablightblue}{HTML}{4DBEEE}
\definecolor{matlabred}{HTML}{A2142F}

\newcommand{\plotfontsize}{\normalfont\footnotesize}
\newcommand{\legendfontsize}{\normalfont\footnotesize}

\tikzstyle{ghostline} = [
  draw = none,
  mark = none
]
\tikzstyle{areafill} = [
  opacity = .6
]
\tikzstyle{lines} = [
  line width = 2pt,
  solid
]
\tikzstyle{linet} = [
  line width = 1pt,
  solid
]
\tikzstyle{col} = [
  line width   = 1.3pt,
  fill opacity = 0.75
]

\pgfplotscreateplotcyclelist{barlist}{
  {draw = matlabblue, fill = matlabblue, col},
  {draw = matlaborange, fill = matlaborange, col},
  {draw = matlabgreen, fill = matlabgreen, col},
  {draw = matlabpurple, fill = matlabpurple, col}
}

\pgfplotscreateplotcyclelist{performancelist}{
  {ghostline},
  {ghostline},
  {matlabblue, areafill},
  {linet, draw = matlabblue, mark = none,},
  {ghostline},
  {ghostline},
  {matlaborange, areafill},
  {linet, draw = matlaborange, mark = none,},
  {ghostline},
  {ghostline},
  {matlabgreen, areafill},
  {linet, draw = matlabgreen, mark = none,},
  {ghostline},
  {ghostline},
  {matlabpurple, areafill},
  {linet, draw = matlabpurple, mark = none,}
}

\pgfplotscreateplotcyclelist{timelist}{
  {ghostline},
  {ghostline},
  {matlabblue, areafill},
  {lines, draw = matlabblue, mark = none,},
  {ghostline},
  {ghostline},
  {matlaborange, areafill},
  {lines, draw = matlaborange, mark = none,},
  {ghostline},
  {ghostline},
  {matlabgreen, areafill},
  {lines, draw = matlabgreen, mark = none,},
  {ghostline},
  {ghostline},
  {matlabpurple, areafill},
  {lines, draw = matlabpurple, mark = none,}
}

\begin{document}


\title{System stabilization with policy optimization on unstable
  latent manifolds}

\author[$\ast$]{Steffen W. R. Werner}
\affil[$\ast$]{%
  Department of Mathematics and Division of Computational Modeling
  and Data Analytics, Academy of Data Science, Virginia Tech, Blacksburg,
  VA 24061, USA.\authorcr
  \email{steffen.werner@vt.edu}, \orcid{0000-0003-1667-4862}%
}

\author[$\dagger$]{Benjamin Peherstorfer}
\affil[$\dagger$]{%
  Courant Institute of Mathematical Sciences, New York University,
  New York, NY 10012, USA.\authorcr
  \email{pehersto@cims.nyu.edu}, \orcid{0000-0002-1558-6775}%
}

\shorttitle{Unstable manifold policy optimization}
\shortauthor{S.~W.~R. Werner, B. Peherstorfer}
\shortdate{2024-07-08}
\shortinstitute{}



\abstract{%
  Stability is a basic requirement when studying the behavior of
  dynamical systems.
  However, stabilizing dynamical systems via reinforcement learning is
  challenging because only little data can be collected over short time
  horizons before instabilities are triggered and data become meaningless.
  This work introduces a reinforcement learning approach that is formulated
  over  latent manifolds of unstable dynamics so that stabilizing policies can
  be trained from few data samples.
  The unstable manifolds are minimal in the sense that they contain the lowest
  dimensional dynamics that are necessary for learning policies
  that guarantee stabilization.
  This is in stark contrast to generic latent manifolds that aim to approximate
  all---stable and unstable---system dynamics and thus are higher dimensional
  and often require higher amounts of data.
  Experiments demonstrate that the proposed approach stabilizes even complex
  physical systems from few data samples for which other methods that operate
  either directly in the  system state space or on generic latent manifolds
  fail.
}

\novelty{}

\maketitle


\section{Introduction}%
\label{sec:intro}

Dynamical systems are core building blocks for describing the time-dependent
behavior of phenomena of interest in science and engineering, including
physical relationships~\cite{BroT11}, engineering systems~\cite{KarMR12}
as well as agents acting in environments in social sciences~\cite{Sil18}.
A major task is controlling dynamical systems so that they are steered towards
desired states.
In this work, we consider the specific control task of stabilizing dynamical
systems, which means that states have to remain in the same neighborhood of the
desired system behavior despite small perturbations~\cite{Loc01, NijV16}.
Stability is an important property because reasonable actions in dynamic
environments should avoid that the system leaves the desired behavior.
In many cases, stability is even a prerequisite for meaningful predictions
about how the system states evolve in the future.
One approach for stabilizing systems is via feedback control, which means that
external inputs are constructed based on the current state of the system such
that the future system states remain within the neighborhood of the desired
steady state (equilibrium point).

Classical control techniques rely on the governing equations of the
dynamics~\cite{BruK19}.
In contrast, data-driven control methods~\cite{BruK19, FliJ13, ZieN42} and
reinforcement learning~\cite{SutB18} learn control policies from data.
Despite immense progress on reinforcement learning and related machine learning
methods for control, stabilizing systems from data remains challenging.
First, recall that the aim is to stabilize dynamical systems and thus it is
reasonable to expect that the open-loop systems without stabilizing policies are
unstable.
Therefore, systems can be queried only for short time horizons before
instabilities are triggered, which makes collecting state trajectories that are
informative about the system challenging.
In fact, intermediate policies during the learning process often lead to
additional destabilization, which makes data collection even more challenging.
Second, the dimension of the parametrization, e.g., the number of weights and
biases of a deep-network parametrization of a policy in reinforcement
learning, typically grows with the complexity of the dynamics described by the
system.
In particular, the number of parameters grows with the state dimension and the
number of control inputs.
This means that the optimization space, over which a policy is trained, can
be rich and thus large amounts of data are typically required to find
informative search directions, which starkly conflicts with the challenge
mentioned above that collecting data samples of unstable systems can be
problematic~\cite{BruPK16, Rec19}.
Third, simulating or querying systems can be expensive.
Either because numerical simulations have to be performed or because physical
experiments need to be conducted.
In any case, each query to collect a data point is expensive and thus typically
only few data points can be collected.
There are many methods for learning surrogate models from data, e.g.,
dynamic mode decomposition~\cite{BruBPetal16, KutBBetal16, Mez05, RowMBetal09,
Sch10, TuRLetal14, WilKR15},
operator inference~\cite{KraPW24, Peh20, PehW16, QiaKPetal20, SawKP23, ShaWK22},
sparse identification methods~\cite{BruPK16, SchTW18},
the Loewner framework~\cite{AntGI16, DrmP22, GosA18, MayA07, SchU16,
SchUBetal18},
and others~\cite{BerP24, KadOFetal21, VlaAUetal22}; however, having only few
data points also means that training such data-driven surrogate models is
challenging and often intractable.

\begin{figure*}[!t]
  \centering
  \resizebox{\linewidth}{!}{%
  \tikzexternalenable%
  \tikzsetnextfilename{overview}%
  \input{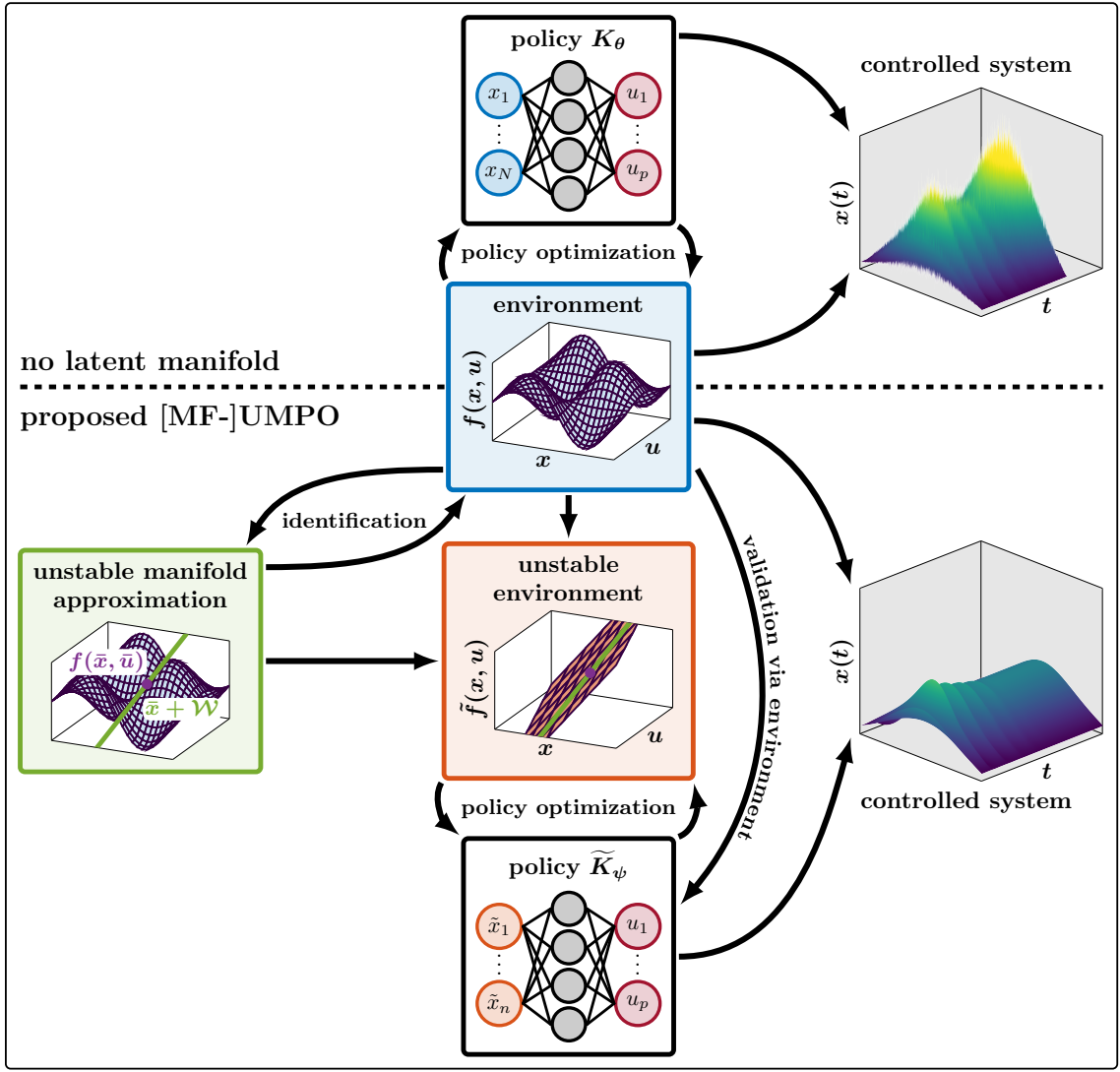}%
  \tikzexternaldisable%
}
  \caption{Reinforcement learning methods that optimize policies without
    unstable manifolds (shown above the dashed line) directly
    query the dynamical environment in its original, high-dimensional state
    representation.
    In contrast, the new \algname{} methods (shown below the dashed line)
    consider the dynamics over the low-dimensional unstable manifold instead,
    which reduces the dimension over which the policy optimization has to
    act and it reduces the complexity of the task at hand by ignoring
    dynamical behavior that is irrelevant for stabilization.}
  \label{fig:overview}
\end{figure*}

In this work, we propose to first learn a low-dimensional manifold on which
exactly those dynamics of the system evolve that are important for
stabilization, and then to learn a policy for stabilizing the dynamics on
that manifold only; see \Cref{fig:overview} for a comparison between our
proposed method and traditional reinforcement learning.
The key insight is that the latent manifolds that are typically used for finding
low-dimensional models in model reduction, e.g., via training
autoencoders~\cite{Bal12, KadBCetal22, KimCWetal22} on state trajectories,
quadratic manifolds~\cite{BarF22, GeeWW23, SchP24}, or
even just computing the principal components~\cite{BenGW15},
are poorly suited for stabilization.
Instead, we restrict the system to the manifold of unstable
dynamics~\cite{Kra91, WerP23a} that captures exactly those dynamics that are
relevant for the task of stabilization.
Because of the low dimensionality of the unstable manifold in many applications
of interest---see~\cite{WerP23a} and \Cref{fig:example_dimensions}---we
can use low-dimensional parametrizations such as deep networks with few
parameters for the policies, which reduce the amount of computational
resources needed for the training process.
Additionally, the learning on unstable manifolds can be combined with
multi-fidelity concepts~\cite{PehWG18} to first pre-train a policy
on a cheap-to-evaluate approximation of the unstable dynamics before the
policy is fine-tuned and certified on the actual system of interest;
see~\cite{WerOP23} for a multi-fidelity optimization approach for robust
control.
The multi-fidelity learning further reduces the computational costs as the
pre-training is independent of the potentially expensive evaluations of the
actual system of interest, which has typically high-dimensional states.

The manuscript is organized as follows.
In \Cref{sec:preliminaries}, we define the setup and provide a problem
formulation.
We then show in \Cref{sec:unstablearn} that latent manifolds that were
trained to generically approximate the system dynamics can be inefficient for
stabilization.
Our approach is introduced in \Cref{sec:mllearn} and numerical
experiments are shown in \Cref{sec:numerics}.
The work is concluded in \Cref{sec:conclusions}.

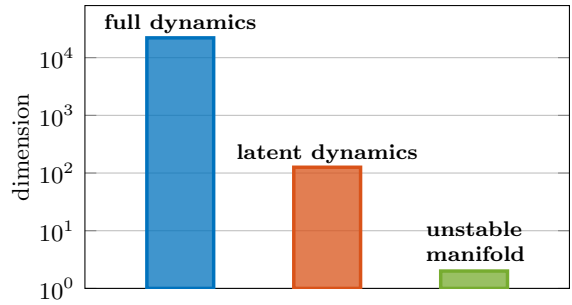
\begin{figure}[t]
  \centering
  \tikzexternalenable%
  \tikzsetnextfilename{dimensions}%
  \begin{tikzpicture}[font = \plotfontsize, text = black]
  \begin{axis}[
    scale only axis,
    width            = .425\linewidth,
    height           = .25\linewidth,
    ybar             = 30pt,
    bar width        = 25pt,
    ymode            = log,
    ymin             = 1,
    ymax             = 8e+04,
    ylabel           = {dimension\vphantom{Pp}},
    ylabel style     = {yshift = -.5em},
    minor tick style = {draw = none},
    ymajorgrids,
    xmin             = 0.9,
    xmax             = 1.1,
    xtick            = {1},
    xticklabels      = {},
    xtick style      = {draw = none},
    cycle list name  = barlist
  ]

    \addplot+[ybar] coordinates{
      (1, 22060)
    };
    \node[xshift = -55pt] at (1, 39000)
      {\scriptsize \textbf{full~dynamics}};

    \addplot+[ybar] coordinates{
      (1, 126)
    };
    \node[xshift = 0pt] at (1, 220)
      {\scriptsize \textbf{latent~dynamics}};

    \addplot+[ybar] coordinates{
      (1, 2)
    };
    \node[xshift = 69pt] at (1, 6.5)
      {\scriptsize
      \begin{minipage}{.15\textwidth}
        \bfseries unstable\\manifold
      \end{minipage}};

  \end{axis}
\end{tikzpicture}%
  \tikzexternaldisable%

  \caption{Unstable manifolds can have much lower dimension that standard
    latent manifolds:
    For the flow past a cylinder described by the Navier-Stokes
    equations, the dimension of the unstable manifold is orders of magnitude
    lower than of a manifold that generically approximates latent dynamics;
    see~\cite{WerP23a}.}
  \label{fig:example_dimensions}
\end{figure}


\section{Stabilizing dynamical systems}%
\label{sec:preliminaries}

In this section, we first introduce the setup for this work before we provide a
problem formulation.


\subsection{Setup}%
\label{sec:setup}

We denote discrete-time dynamical systems as
\begin{equation} \label{eqn:sys}
  x(t + 1) = f(x(t), u(t)), \qquad t \in \N \cup \{ 0 \},
\end{equation}
with initial condition $x(0) \in \Xcal_{0} \subseteq \R^{\nh}$,
$\nh$-dimensional internal states $x(t) \in \R^{\nh}$ and
$\np$-dimensional external control $u(t) \in \R^{\np}$.
The function $f$ can be nonlinear in both arguments.
The task that we consider here is the construction of a policy $K$
that maps the state $x(t)$ to an input $u(t)$  so that the
system~\cref{eqn:sys} is locally stabilized towards a desired equilibrium
point $(\xs, \us)$ for which we have $\xs = f(\xs, \us)$.
That is, the closed-loop system
\begin{equation*}
  x(t + 1) = f(x(t), K(x(t)))
\end{equation*}
converges to the steady state $x(t) \to \xs$ in the limit $t \to \infty$, if the
initial state $x(0)$ lies close enough to the steady state $\xs$ with
$\lVert x(0) - \xs \rVert \leq \epsilon$ for an $\epsilon > 0$.
Notice that in the context of control, the policy $K$ is typically referred to
as state-feedback controller.
For the construction of the policy $K$, we can query the
system~\cref{eqn:sys} at admissible initial conditions $x(0) \in \Xcal_{0}$
and control trajectories $u(0), u(1), \ldots, u(\tf - 1)$ to observe
state trajectories $x(1), x(2), \ldots, x(\tf)$ up to some final time $\tf$.


\subsection{Problem formulation}%
\label{sec:problem}

For stabilization via reinforcement learning, we introduce a reward function
$r\colon \R^{\nh} \times \R^{\np} \to \R$ so that $r(x(t), u(t))$ penalizes the
distance of $x(t)$ and $u(t)$ from the considered steady state $(\xs, \us)$.
We parametrize the policy $\Kt\colon \R^{\nh} \to \R^{\np}$ over the set of
$\nm$-dimensional parameters $\Theta \subseteq \R^{\nm}$ and we then seek an
$\nm$-dimensional parameter $\theta_{\ast} \in \Theta$ that maximizes the
accumulated reward:
\begin{equation}\label{eqn:Prelim:RLObj}
  \theta_{\ast} \in \arg\max\limits_{\theta \in \Theta}
    \sum\limits_{t = 0}^{\tf} r\big( x(t),\Kt(x(t)) \big).
\end{equation}
This is a challenging problem.
First, because the task is stabilization, we reasonably expect the open-loop
system~\cref{eqn:sys} without policy to be unstable and therefore it can
be queried only for a short time before instabilities are triggered.
Thus, collecting state trajectories is limited due to the instability.
In fact, policies $\Kt$ corresponding to intermediate parameters $\theta$
during the optimization process typically lead to destabilization,
which makes data collection even more challenging.
Second, simulating or querying the system~\cref{eqn:sys} can be expensive.
Third, the number of components $\nm$ of $\theta$ typically grows with the
complexity of the dynamics described by the system~\cref{eqn:sys}; in
particular with the state dimension~$\nh$ and the number of
control inputs~$\np$.
Therefore, the search space $\Theta$ can be rich and thus large amounts of data
are typically required to find informative search directions.


\section{Learning on the manifold of unstable dynamics}%
\label{sec:unstablearn}

We propose to learn a low-dimensional manifold in which the dynamics of
the system~\cref{eqn:sys} evolve that are important for stabilization, and then
to learn a policy $\Kt$ for dynamics on the manifold only.
The key insight is that latent manifolds that are typically used for finding
low-dimensional latent models and model reduction (e.g., via training
autoencoders on state trajectories or even just computing the principal
components) are poorly suited for stabilization and instead one should restrict
the system to the manifold of unstable dynamics that captures exactly those
dynamics that are relevant for the task of stabilization.


\subsection{Stabilization on latent manifolds}%
\label{sec:manifolds}

We denote a latent manifold of the $\nh$-dimensional state space of
system~\cref{eqn:sys} as
\begin{equation} \label{eqn:latentm}
  \Mcal = \left\{ (D \circ E)(z)\colon z \in \R^{\nh} \right\},
\end{equation}
which is defined via the encoder $E\colon \R^{\nh} \to \R^{\nr}$ and
decoder $D\colon \R^{\nr} \to \R^{\nh}$.
To give a brief outlook, in \Cref{sec:rlstab}, we will learn a policy
$\Krr: \R^{\nr} \to \R^{\np}$ that acts on $\nr$-dimensional elements of
the encoded manifold $E(\Mcal)$ such that $\Ktr = \Krr \circ E$ stabilizes the
high-dimensional system~\cref{eqn:sys};
at least locally about the steady state $(\xs, \us)$.
The parametrization $\thetar \in \Thetar \subseteq \R^{\nmr}$ of $\Krr$ is
independent of the dimension $\nh$ of the high-dimensional state $x(t)$ and
instead depends on the dimension $\nr$ of the manifold $\Mcal$.
In the case that $\nr \ll \nh$, this reduces the search space of admissible
parametrizations from $\Theta$ with dimension $\nm$ to $\Thetar$ with dimension
$\nmr \ll \nm$ and, consequently, the computational costs of learning $\Ktr$.

In \Cref{sec:mrlstab}, we will go one step further and additionally
restrict the system dynamics of~\cref{eqn:sys} onto $\Mcal$ by considering the
states $\xr(t) \in \R^{\nr}$ with dimension $\nr \ll \nh$ given by the latent
model
\begin{equation} \label{eqn:ManifoldModel}
  \xr(t + 1) = E \circ f(D(\xr(t)), u(t)),
\end{equation}
Then, the policy $\Krr$ will be trained on the latent
model~\cref{eqn:ManifoldModel}, which avoids having to query the potentially
more expensive, high-dimensional model~\cref{eqn:sys}.
However, additional conditions on $E$ and $D$ are necessary such that the
dynamics described by~\cref{eqn:ManifoldModel} are informative enough for the
stabilization of~\cref{eqn:sys} by $\Ktr = \Krr \circ E$.
Additionally, due to computational inexactness, limited data and the
finite-dimensional parametrizations of the complex manifolds that hold the
dynamics important for stabilization, one can trust the latent
model~\cref{eqn:ManifoldModel} in actual numerical computations only in limited
cases and, therefore, recourse to the high-dimensional model~\cref{eqn:sys}
will certify the stabilization;
see multi-fidelity methods~\cite{PehWG18, WerOP23}.


\subsection{Why not every latent manifold is a good choice for stabilization}%
\label{sec:badman}

We now discuss that latent manifolds $\Mcal$ that describe well essential
system dynamics so that the states $\xr(t)$ of the latent
model~\cref{eqn:ManifoldModel} approximate the states $x(t)$ of the
high-dimensional model~\cref{eqn:sys} can still lead to catastrophic errors
when used for stabilization.
A well-known approach to compute suitable latent manifolds to approximate the
system dynamics in the special case of linear systems is the balanced
truncation method~\cite{Moo81}.
However, the system needs to be known for this method and only under strict
assumptions, these manifolds are suitable for the task of
stabilization~\cite{BenHW22, MusG91}.
In the following, we consider the sampling-based approach of principal
components, which can lead to good approximations of the system dynamics in
general but either ignore or overfit the unstable behavior.


\begin{figure*}[t]
  \centering
  \begin{subfigure}[b]{.49\linewidth}
    \centering
  \tikzexternalenable%
  \tikzsetnextfilename{example_stabeig3}%
  \begin{tikzpicture}[font = \plotfontsize, text = black]
  \pgfplotstableread{graphics/data/example_stabeig3.dat}\tableDATA

  \begin{axis}[%
    scale only axis,
    width           = .775\linewidth,
    height          = .35\linewidth,
    xmin            = -0.2,
    xmax            = 0.0,
    ymin            = 0.9,
    ymax            = 1.15,
    xminorticks     = false,
    yminorticks     = false,
    scaled x ticks  = false,
    xticklabels     = {x, -0.2, -0.15, -0.1, -0.05, 0},
    xlabel          = {controller parameter $\thetar$},
    ylabel          = {eigenvalue magnitude},
    ylabel style    = {yshift = -.4em},
    xlabel style    = {yshift = .4em},
    legend columns  = -1,
    legend style    = {
      at     = {(.5, -.35)},
      font   = \legendfontsize,
      anchor = north,
      /tikz/every even column/.append style = {column sep = .05in}},
    legend cell align = {left},
    clip mode         = individual
  ]

    \addplot[
      lines,
      matlabblue,
      mark         = *,
      mark options = {fill = matlabblue},
      mark phase   = 0,
      mark repeat  = 38
    ] table[x index = 0, y index = 1]{\tableDATA};
    \addlegendentry{\scriptsize $\lambdacr$}

    \addplot[
      lines,
      matlaborange,
      mark         = square*,
      mark options = {fill = matlaborange},
      mark phase   = 0,
      mark repeat  = 38
    ] table[x index = 0, y index = 3]{\tableDATA};
    \addlegendentry{\scriptsize $\lambdacu$}

    \addplot[
      lines,
      matlabgreen,
      mark         = triangle*,
      mark options = {fill = matlabgreen},
      mark phase   = 0,
      mark repeat  = 38
    ] table[x index = 0, y index = 2]{\tableDATA};
    \addlegendentry{\scriptsize $\lambdacs$}

    \addplot[lines, black, dashed] coordinates{(-0.2, 1.0) (0.0, 1.0)};
    \addlegendentry{\scriptsize stability border}
  \end{axis}
\end{tikzpicture}%
  \tikzexternaldisable%

    \vspace{-\baselineskip}
    \caption{High-dimensional and latent controlled system modes.}
    \label{fig:example_pca_eig}
  \end{subfigure}%
  \hfill%
  \begin{subfigure}[b]{.49\linewidth}
    \centering
  \tikzexternalenable%
  \tikzsetnextfilename{example_stabeig3_sim}%
  \begin{tikzpicture}[font = \plotfontsize, text = black]
  \begin{axis}[
    scale only axis,
    width                  = .76\linewidth,
    xmin                   = -0.2,
    xmax                   = 0,
    ymin                   = 0,
    ymax                   = 100,
    zmin                   = -0.1946,
    zmax                   = 3.8797,
    grid                   = none,
    xlabel                 = {$\thetar$},
    ylabel                 = {$t$},
    zlabel                 = {$\log_{10}(\lVert x(t) -\xs \rVert_{2})$},
    xlabel style           = {yshift = 1em},
    ylabel style           = {yshift = 1em},
    zlabel style           = {yshift = .25em},
    axis background/.style = {fill = black!10}
  ]
    \addplot3 graphics[points = {
      (0, 2, -0.1031) => (331.9, 237.4-235.4)
      (-0.2, 0, -0.0431) => (2, 237.4-174.7)
      (-0.198, 85, 0.3915) => (278.6, 237.4-105.1) 
      (0, 100, 3.8796) => (649, 237.4-2.2)
    }]{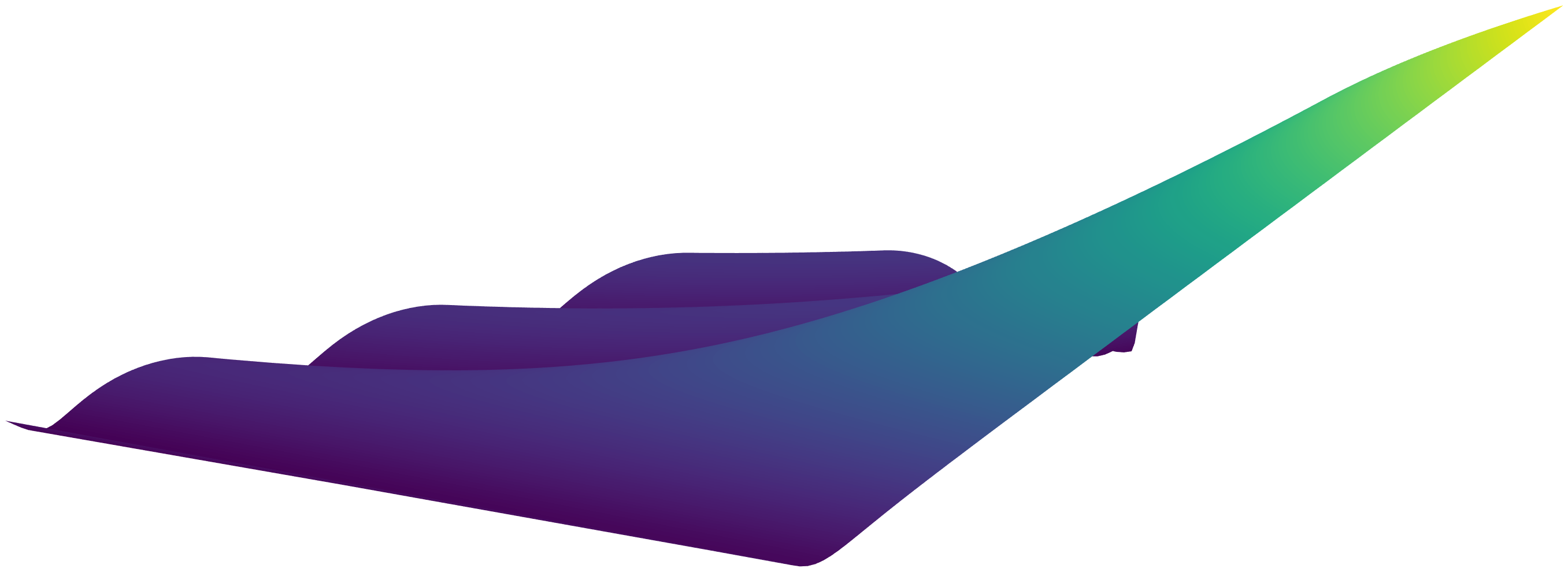};
  \end{axis}
\end{tikzpicture}%
  \tikzexternaldisable%

    \vspace{0\baselineskip}
    \caption{Growing log-magnitudes of controlled simulations.}
    \label{fig:example_pca_sim}
  \end{subfigure}

  \caption{PCA subspaces are insufficient for stabilization:
    While the latent model created from a fully converged low-dimensional PCA
    subspace has the same unstable eigenvalue as the true system, its controlled
    behavior does not coincide  with the true system steered
    via the corresponding decoded policy, since the controlled
    latent eigenvalue $\lambdacr$ does not change, while the controlled
    true eigenvalues $\lambdacu$ and $\lambdacs$ do (see (a)).
    Tuning the policy parameter $\thetar$ does not stabilize
    the true system since $\lambdacu$ and $\lambdacs$ do not decrease
    below the stability border and the simulation trajectories (see (b))
    do not converge to the steady state but oscillate or even grow in
    magnitude.}
  \label{fig:example_pca}
\end{figure*}
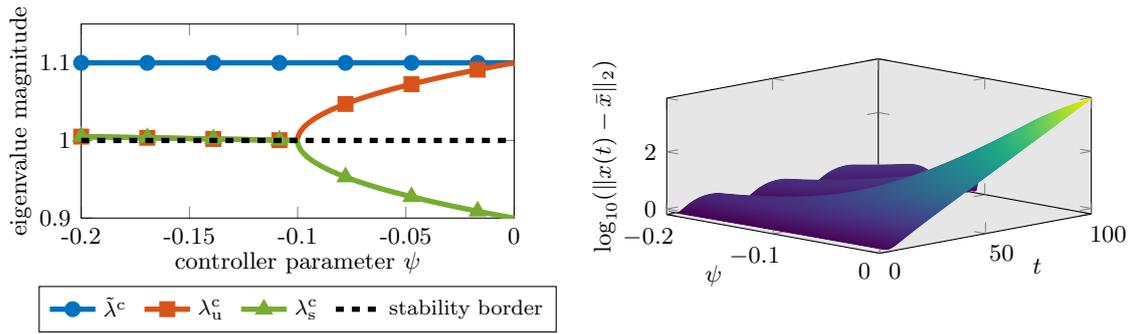

\paragraph{Principal components can lead to unsuitable latent manifolds for
stabilization.}%
A classical approach to find a suitable manifold $\Mcal$ for the system dynamics
of~\cref{eqn:sys} is to train the manifold so that the embedding error
$\lVert (D \circ E)(x(t)) - x(t) \rVert$ is small.
We consider the linear manifold of dominant dynamics $\Vcal$ that minimizes the
embedding error, which is given by the principal component analysis
(PCA)~\cite{Pea01, Sch10, TuRLetal14}.
We collect the principal components as orthogonal columns in
$V \in \R^{\nh \times \nr}$ to obtain the encoder and decoder as
\begin{equation*}
  D_{\Vcal}(\xr(t)) = V \xr(t) \quad\text{and}\quad
  E_{\Vcal}(x(t)) = V^{\trans} x(t).
\end{equation*}
We use principal components just for ease of exposition here in this motivating
example but the insights apply to nonlinear embeddings as well.
Consider the linear system $x(t + 1) = f(x(t), u(t)) = A x(t) + B u(t)$ given by
\begin{equation} \label{eqn:example}
  A = \begin{bmatrix} 0.9 & 0 \\ \varepsilon & 1.1 \end{bmatrix}, \quad
  B = \begin{bmatrix} 1 \\ 0 \end{bmatrix}, \quad
  x(0) = \begin{bmatrix} 0 \\ 0 \end{bmatrix},
\end{equation}
with $\varepsilon > 0$.
The system in~\cref{eqn:example} has one stable and one unstable eigenvalue
(system mode), where the parameter $\varepsilon$ allows to adjust the strength
of the coupling between the corresponding dynamics.
Mathematically, the strength of the coupling corresponds to the angle
between the left and right eigenspaces of the unstable eigenvalue.
A small angle means that principal components mostly ignore the directions in
which the system gets unstable; a larger angle means that unstable directions
are captured by principal components but result in unwanted influence of the
policy on the stable directions, even leading to destabilization
of that system component.
The policy for the latent model is parametrized via
$\Krr(\xr) = \thetar \xr$, with the policy parameter $\thetar \in \R$.
For $\varepsilon = 0.1$, \Cref{fig:example_pca} illustrates that
no policy formulated over the PCA subspace is capable of stabilizing the true
system~\cref{eqn:example} of interest.
This can be seen clearly in \Cref{fig:example_pca_sim} where for none of the
policy parameters $\thetar$, the trajectories converge to the steady state
$\xs$ but rather oscillate or even diverge indicated by the growing
magnitudes.
The underlying problem with the PCA-based manifold is revealed in
\Cref{fig:example_pca_eig}.
The eigenvalues of the controlled systems shown here are denoted by
\begin{equation*}
  \lambdacr = \Lambda(\Jxr + \Jur \Krr(1)) \quad\text{and}\quad
    \{ \lambdacs, \lambdacu \} = \Lambda(A + B \Krr \circ E_{\Vcal}(I_{2})),
\end{equation*}
where $\Jxr = V^{\trans} A V$, $\Jur = V^{\trans} B$,
$E_{\Vcal}(I_{2})$ denotes the column-wise application of $E_{\Vcal}$ to the
two-dimensional identity matrix $I_{2}$, and $\lambdacs$ and $\lambdacu$
correspond to the stable and unstable eigenvalues of the original
system~\cref{eqn:example}, respectively.
While the corresponding latent model that lives on $\Mcal$ has the same
unstable system mode as the original system~\cref{eqn:example}, the influence
of policies on the latent model corresponding to PCA appears to be strongly
different from its effect on the true system.
In fact, there is no visible area in \Cref{fig:example_pca_eig} in which
both the controlled eigenvalues $\lambdacu$ and $\lambdacs$ decrease below the
stability border, which would indicate the stability of the true controlled
system by the low-dimensional policy, and the latent model mode $\lambdacr$ is
stabilized in a different region of the parameter $\thetar$ that is not shown in
\Cref{fig:example_pca_eig}.
Due to monotinicity of the system modes, the true system of interest cannot be
stabilized in the parameter region, where the latent model is stabilized, and
vice versa.


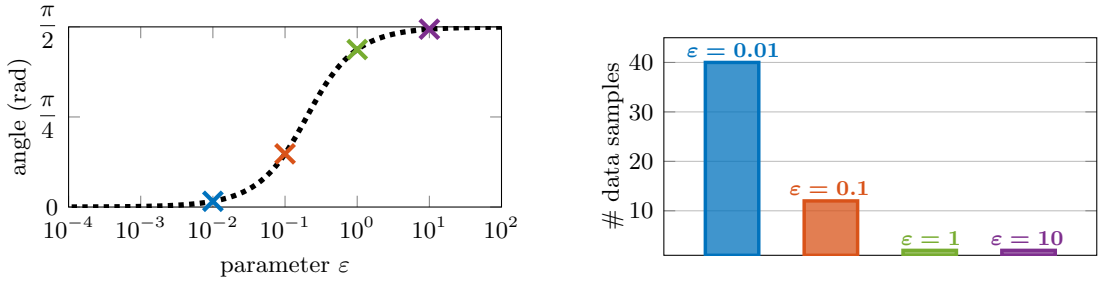
\begin{figure*}[t]
  \centering
  \begin{subfigure}[b]{.49\linewidth}
    \centering
  \tikzexternalenable%
  \tikzsetnextfilename{example_angles}%
  \begin{tikzpicture}[font = \plotfontsize, text = black]
  \pgfplotstableread{graphics/data/example_angles.dat}\tableDATA
  
  \begin{semilogxaxis}[%
    scale only axis,
    width           = .775\linewidth,
    height          = .325\linewidth,
    xmin            = 1e-4,
    xmax            = 1e+2,
    ymin            = 0,
    ymax            = 1.5708,
    ytick           = {0, 0.7854, 1.5708},
    yticklabels     = {$0$,
                       $\displaystyle \frac{\pi}{4}$,
                       $\displaystyle \frac{\pi}{2}$},
    xminorticks     = false,
    yminorticks     = false,
    scaled x ticks  = false,
    xlabel          = {parameter $\varepsilon$},
    ylabel          = {angle (rad)},
    ylabel style    = {yshift = -.5em}
  ]

    \addplot[lines, black, dotted] table[x index = 0, y index = 1]{\tableDATA};

    \addplot[mark = x, mark options = {scale = 2.5, ultra thick}, ultra thick,
      only marks, matlabpurple] coordinates {(10, 1.5508)};

    \addplot[mark = x, mark options = {scale = 2.5, ultra thick}, ultra thick,
      only marks, matlabgreen] coordinates {(1, 1.3734)};

    \addplot[mark = x, mark options = {scale = 2.5, ultra thick}, ultra thick,
      only marks, matlaborange] coordinates {(0.1, 0.4636)};

    \addplot[mark = x, mark options = {scale = 2.5, ultra thick}, ultra thick,
      only marks, matlabblue] coordinates {(0.01, 0.0500)};
  \end{semilogxaxis}
\end{tikzpicture}%
  \tikzexternaldisable%

    \caption{Distance between PCA and unstable manifold.}
    \label{fig:example_angles}
  \end{subfigure}%
  \hfill%
  \begin{subfigure}[b]{.49\linewidth}
    \centering
  \tikzexternalenable%
  \tikzsetnextfilename{example_data}%
  \begin{tikzpicture}[font = \plotfontsize, text = black]
  \begin{axis}[
    scale only axis,
    width            = .775\linewidth,
    height           = .3925\linewidth,
    ybar             = 17pt,
    bar width        = 20pt,
    ymin             = 1,
    ymax             = 45,
    ylabel           = {\# data samples\vphantom{Pp}},
    ylabel style     = {yshift = -.5em},
    minor tick style = {draw = none},
    ymajorgrids,
    xmin             = 0.9,
    xmax             = 1.1,
    xtick            = {1},
    xticklabels      = {},
    xtick style      = {draw = none},
    cycle list name  = barlist
  ]

    \addplot+[ybar] coordinates{
      (1, 40)
    };
    \node[xshift = -55pt] at (1, 42.5)
      {\scriptsize \textcolor{matlabblue}{$\boldsymbol{\varepsilon = 0.01}$}};

    \addplot+[ybar] coordinates{
      (1, 12)
    };
    \node[xshift = -19pt] at (1, 14.5)
      {\scriptsize \textcolor{matlaborange}{$\boldsymbol{\varepsilon = 0.1}$}};
    
    \addplot+[ybar] coordinates{
      (1, 2)
    };
    \node[xshift = 19pt] at (1, 4.5)
      {\scriptsize \textcolor{matlabgreen}{$\boldsymbol{\varepsilon = 1}$}};

    \addplot+[ybar] coordinates{
      (1, 2)
    };
    \node[xshift = 55pt] at (1, 4.5)
      {\scriptsize \textcolor{matlabpurple}{$\boldsymbol{\varepsilon = 10}$}};
    
  \end{axis}
\end{tikzpicture} %
  \tikzexternaldisable%

    \caption{Number of data samples for PCA to find instability.}
    \label{fig:example_data}
  \end{subfigure}

  \caption{Increasing the distance between the PCA and the manifold of
    unstable system dynamics makes the PCA manifold less suited for the task
    of stabilization.
    Via the dynamics coupling in this example~\cref{eqn:example}, the number of
    data samples needed to identify any instabilities on the PCA manifold
    increases by $20\times$ from $\varepsilon = 0.01$ to $\varepsilon = 10$.}
  \label{fig:example_dist_data}
\end{figure*}

\paragraph{Large amounts of data are required to compensate for poorly suited
latent manifolds.}%
Not targeting the unstable dynamics directly translates into requiring large
amounts of trajectory data to capture the dynamics necessary for stabilization.
Even when PCA captures the unstable dynamics eventually, building a policy
in the PCA subspace can lead to a pollution effect in the sense that the learned
policy also affects stable dynamics in undesired ways as shown in
\Cref{fig:example_pca}; see~\cite{WerP24}.
\Cref{fig:example_dist_data} shows the distance between the PCA-based
manifold of system dynamics and the manifold of unstable dynamics, which is
needed for the design of stabilizing policies, and the corresponding amounts
of data needed for PCA to identify any instability.
With increasing distance between the two manifolds, given in terms of
subspace angles in \Cref{fig:example_angles}, the PCA becomes less suitable
for the design of stabilizing policies as the amount of data needed to
identify instabilities increases; see \Cref{fig:example_data}.
Sometimes this can be compensated using even more data samples in the learning
process of the policy, but there are no guarantees as, for example,
\Cref{fig:example_pca} indicates that there is no stabilizing policy on
the PCA manifold.


\subsection{The latent manifold of unstable dynamics}%
\label{sec:umanifold}

We propose to consider the dynamics relevant for the task of stabilization,
which means the purely unstable system dynamics, instead of aiming to find a
latent manifold $\Mcal$ that approximates all system dynamics in general.
The manifold of unstable dynamics about the steady state $(\xs, \us)$ is
defined as
\begin{equation} \label{eqn:unstabman}
  \Mcalu(\xs, \us) = \left\{ x(t) \in \R^{\nh}\colon
    \lim\limits_{t \to \infty} \phi^{-t}(x(t)) = \xs ~\text{for}~
    u(t) = \us \right\},
\end{equation}
where $\phi^{-t}$ is the inverse flow map of~\cref{eqn:sys} that goes backwards
in time~\cite{KraODetal05, ZieDG19}.
Thus, the manifold $\Mcalu(\xs, \us)$ includes the trajectories that
originate arbitrarily close to the steady state $(\xs, \us)$ but can
diverge from it over time.
For the low-dimensional representation of the unstable
manifold~\cref{eqn:unstabman}, we need to find encoder $E$ and decoder $D$ that
fully characterize its dynamics.
Therefore, the time evolution in~\cref{eqn:ManifoldModel} is the restriction of
the states $x(t)$ such that $\xr(t) = E(x(t))$ holds for all $x(t) \in \Mcalu$.
This exact representation yields the equation
\begin{equation*}
  E(x(t+1)) = E \circ f((D \circ E)(x(t)), u(t)).
\end{equation*}
Inserting this into the encoding of the high-dimensional
system~\cref{eqn:sys} results in the following condition to hold for the
encoder and decoder in terms of the system dynamics and states:
\begin{equation} \label{eqn:autoinvar}
  E \circ f(x(t), u(t)) = E \circ f((D \circ E)(x(t)), u(t)),
\end{equation}
for all $x(t) \in \Mcalu(\xs, \us)$.
Since we consider local stabilization with the states $x(t) \in \Mcalu$
satisfying $\lVert x(t) - \xs \rVert \leq \epsilon$ for some $\epsilon > 0$,
and under the assumption that the function $f$ in~\cref{eqn:sys} is analytic in
both arguments, the linearization of $f$ about $(\xs, \us)$ given via its
Taylor series expansion will yield sufficient information about the dynamics:
\begin{equation*}
  f(x(t), u(t)) = f(\xs, \us) + \Jx(x(t) - \xs) + \Ju(u(t) - \us) +
    \Ocal\big( (x(t) - \xs)^{2} \big) + \Ocal\big( (u(t) - \us)^{2} \big),
\end{equation*}
where $\Jx(x) = (\nabla_{x}f(\xs, \us)) x$ is the Jacobian of
$f$ with respect to the state and
$\Ju(u) = (\nabla_{u}f(\xs, \us))u$ is the Jacobian of $f$ with
respect to the input.
Since $(\xs, \us)$ is an equilibrium point, we have that $\xs = f(\xs, \us)$.
Therefore, by introducing the new state and input variables
$\xc(t) = x(t) - \xs$ and $\uc(t) = u(t) - \us$ that describe the deviation
from the desired system behavior in~\cref{eqn:sys} yields
\begin{equation*}
  x(t + 1) - \xs
    = \xc(t + 1)
    = \Jx(\xc(t)) + \Ju(\uc(t)) + \Ocal\big( \xc(t)^{2} \big) +
      \Ocal\big( \uc(t)^{2} \big).
\end{equation*}
About $(\xs, \us)$, the original system~\cref{eqn:sys} is well approximated by
\begin{equation} \label{eqn:linear}
  \xc(t + 1) = \Jx(\xc(t)) + \Ju(\uc(t)).
\end{equation}
Note that $(0, 0)$ is the steady state of~\cref{eqn:linear} that corresponds
to $(\xs, \us)$ in~\cref{eqn:sys}.
Inserting~\cref{eqn:linear} into the condition~\cref{eqn:autoinvar} and making
the linear ansatz
\begin{equation} \label{eqn:unstabencode}
  E_{\Wcal}(x(t)) = W^{\trans} x(t) ~~\text{and}~~
  D_{\Wcal}(\xr(t)) = (W^{\trans})^{\dagger} \xr(t),
\end{equation}
with $W \in \R^{\nh \times \nr}$, for the encoder and decoder results in the
invariant subspace equation
\begin{equation} \label{eqn:eig}
  W^{\trans} \Jx = \left( W^{\trans} \Jx (W^{\trans})^{\dagger}
    \right) W^{\trans},
\end{equation}
where the eigenvalues of $(W^{\trans} \Jx (W^{\trans})^{\dagger})$ are all
the eigenvalues of $\Jx$ with absolute value larger or equal to one.
The basis matrix $W$ spans the subspace $\Wcal$ such that $\xs + \Wcal$
yields an exact description of $\Mcalu$ in the steady state of
interest $(\xs, \us)$ and is expected to be a suitable approximation for
the rest of the dynamics  $x(t) \in \Mcalu(\xs, \us)$ with
$\lVert x(t) - \xs \rVert \leq \epsilon$;
see~\cite{KraODetal05, WerP23a, ZieDG19}.
Also, note that the subspace $\Wcal$ has the same dimension $\nr$ as the
manifold $\Mcalu(\xs, \us)$, which only increases with the number of unstable
eigenvalues of the Jacobian $\Jx$ in the steady state.
In other words, the dimension only depends on the number of unstable system
modes, which is independent of the state dimension $\nh$ and the dimension
$\np$ of the controls; see \Cref{fig:example_dimensions}.


\subsection{Estimation of encoders and decoders from adjoint data}%
\label{sec:unstabest}

For the construction of the encoder $E_{\Wcal}$ in~\cref{eqn:unstabencode} from
data, we observe that~\cref{eqn:eig} defines the basis matrix $W$ via a linear
left eigenvalue equation in matrix form~\cite{GolV13}.
Therefore, the unstable manifold about the steady state is described by the
left eigenvectors corresponding to the unstable eigenvalues of $\Jx$.
The estimation of the left eigenvalues and eigenvectors of the Jacobian of $f$
is possible via adjoint data, which can be obtained in a non-intrusive
way~\cite{GhaW21, KraPW24, PehW16} via automatic differentiation of $f$ with
respect to $x(t)$.
In particular, the eigenvectors are estimated using data obtained by querying
the adjoint system on initial conditions and observing the state trajectories.
Many modern code packages allow for the use of automatic differentiation like
JAX~\cite{BraFHetal18},
FEniCS DOLPHIN~\cite{MitFD19},
TensorFlow~\cite{AbaBCetal16} and
PyTorch~\cite{PasGMetal19}.
To generate data from the adjoint system, we consider the idea of reverse
accumulation also known as adjoint mode automatic differentiation.
Reverse accumulation allows for vector-Jacobian products of the form
\begin{equation*}
  F[\xs, \us](z) = z^{\trans} \Jx = z^{\trans} \nabla_{x} f(\xs, \us),
\end{equation*}
for $z \in \R^{N}$, which  is equivalent to querying the adjoint system given
by the adjoint operator
\begin{equation} \label{eqn:jacop}
  G[\xs, \us](z) = \left( F[\xs, \us](z) \right)^{\trans} = \Jx^{\trans} z
    = \left( \nabla_{x} f(\xs, \us) \right)^{\trans} z.
\end{equation}
The left eigenspaces of $\Jx$ given by the operator $F$ are the
right eigenspaces of $\Jx^{\trans}$ given by the operator $G$.
Notice that the eigenvectors are computed using only queries of
the operator on vectors, which are given by state observations.
Thereby, we can compute the left eigenspaces in~\cref{eqn:eig} by
evaluating the adjoint system~\cref{eqn:jacop} using vector-Jacobian products
in Krylov methods~\cite{Arn51, Ste01}.


\section{Leveraging the latent manifold of unstable dynamics for
reinforcement learning}%
\label{sec:mllearn}

In this section, we introduce unstable manifold policy optimization (UMPO) to
optimize for a policy that acts on the latent states of the manifold of
unstable dynamics.
We further introduce multi-fidelity policy optimization (MF-UMPO) that combines
optimizing for a policy with a latent model over the unstable manifold and
subsequently correcting the policy with the high-dimensional state model.


\subsection{Policy optimization on unstable manifolds}%
\label{sec:rlstab}

We now leverage the approximation of the unstable
latent manifold $\xs + \Wcal$ to train a policy $\Ktr$ with a
low-dimensional parametrization for stabilizing the high-di\-men\-sio\-nal
system~\cref{eqn:sys}.
We denote this method as \Proj{}.
Thereby, we first learn $\Krr: \R^{\nr} \to \R^{\np}$ that is
parametrized over $\Thetar \subseteq \R^{\nmr}$ and that acts on latent
states of dimension $\nr$.
Then, we lift $\Krr$ via the encoder $E_{\Wcal}$ of the unstable latent
manifold $\Mcalu$ to act on states $x(t)$ of the high-dimensional
system~\cref{eqn:sys},
\begin{equation} \label{eqn:projneural}
  \Ktr(x(t)) = \Krr( E_{\Wcal}(x(t) - \xs)) + \us,
\end{equation}
where a shift by the equilibrium point $(\xs, \us)$ is added to center the
unstable manifold $\Mcalu$ at the steady state $\xs$ and the control signals at
the steady-state control $\us$.
The parameter vector $\thetar$ of $\Ktr$ can then be trained with an
objective such as~\cref{eqn:Prelim:RLObj} by querying the high-dimensional
system~\cref{eqn:sys} using, e.g., policy gradient methods
(DDPG~\cite{LilHPetal15}, TD3~\cite{FujVM18})
that are suited for continuous observation and action spaces.

In contrast to direct approaches that learn $\Kt$ by acting on the
$\nh$-dimensional states $x(t)$ of the high-dimensional system~\cref{eqn:sys},
the policy $\Ktr$ given in~\cref{eqn:projneural} includes an encoding step
with $E_{\Wcal}$ that maps the shifted $x(t)$ onto a low-dimensional latent
state of dimension~$\nr$.
Thus, the parametrization $\Thetar \subseteq \R^{\nmr}$ of $\Ktr$ depends on
the dimension of the unstable latent manifold $\nr$, in contrast to
$\Theta \subseteq \R^{\nm}$ of $\Kt$ that depends on the dimension $\nh$
of the high-dimensional model~\cref{eqn:sys}.
Due to the change in the dimensions of the policy mapping, a lower
dimensional parametrization $\nmr \ll \nm$ is sufficient for $\Ktr$ than for
$\Kt$, which reduces the optimization space and training costs, as we will
demonstrate with numerical experiments below.
The \Proj{} method is summarized in \Cref{alg:MRLStab} for the case of
\emph{no multi-fidelity} using
\Cref{alg:MRLStab_eig,alg:MRLStab_pre2,alg:MRLStab_train}.


\subsection{Multi-fidelity policy optimization on the unstable manifold}%
\label{sec:mrlstab}

We now propose a multi-fidelity version of \Proj{}, where an policy is
pre-trained on a cheap-to-evaluate surrogate latent model of the system
dynamics of interest so that the pre-trained policy is a good starting point for
easing the training on the high-dimensional system~\cref{eqn:sys}.
Additionally, because the agent is trained on the high-dimensional model
eventually, confidence in the learned policy can be higher than when
learning with a surrogate model alone.


\paragraph{Latent model of unstable dynamics for pre-training.}%
The latent model of unstable dynamics is given by
\begin{equation} \label{eqn:rom}
  \xr(t + 1) = \Jxr(\xr(t)) + \Jur(\ur(t)),
\end{equation}
with the state transition and control terms
\begin{equation} \label{eqn:jaclin}
  \Jxr(\xr(t)) = E_{\Wcal}(\nabla_{\operatorname{x}} f(\xs, \us)
    D_{\Wcal}(\xr(t)))
  \quad \text{and} \quad
  \Jur(\ur(t)) =  E_{\Wcal}(\nabla_{\operatorname{u}} f(\xs, \us) \ur(t)),
\end{equation}
where the encoder $E_{\Wcal}$ and decoder $D_{\Wcal}$ are defined
in~\cref{eqn:unstabencode}.
This is a specific instance of the generic latent model given
in~\cref{eqn:ManifoldModel} that results from applying the chosen encoder and
decoder~\cref{eqn:autoinvar} to the linearized dynamics in~\cref{eqn:linear}.
As $\Jxr$ and $\Jur$ in~\cref{eqn:rom} are linear, the describing matrices can
be assembled using the reverse mode automatic differentiation as in the
estimation of $\Wcal$ in \Cref{sec:unstabest}.
The terms~\cref{eqn:jaclin} can then be computed by
\begin{equation} \label{eqn:reverseid}
  \Jxr = E_{\Wcal}(\nabla_{\operatorname{x}} f(\xs, \us)) (W^{\trans})^{\dagger}
  \quad \text{and} \quad
  \Jur = E_{\Wcal}(\nabla_{\operatorname{u}} f(\xs, \us)),
\end{equation}
where $E_{\Wcal}(.)$ above is applied column-wise to the argument and
$(W^{\trans})^{\dagger}$ comes from the decoder $D_{\Wcal}$
in~\cref{eqn:unstabencode}.
Alternatively, the terms in~\cref{eqn:rom} can be learned via $\nr + \np$
forward evaluations of~\cref{eqn:sys} and system identification in the subspace
$\Wcal$; see~\cite[Prop.~1]{WerP24}.
This can be advantageous compared to~\cref{eqn:reverseid} because it allows to
separate the computational steps for estimating $\Wcal$
and~\cref{eqn:rom}, which is important in modularized implementations.
In any case, the terms in~\cref{eqn:jaclin} for~\cref{eqn:rom} can be
precomputed and stored as matrices, which means that querying the latent
model of unstable dynamics~\cref{eqn:rom} is cheap, and thus well suited
for pre-training in a first step of multi-fidelity reinforcement learning.


\paragraph{Multi-fidelity fine-tuning on the unstable manifold.}%
We now propose multi-fidelity policy optimization with the unstable latent
manifold (\ROMProj{}), which proceeds in two steps.
In Step~1, the policy $\Krr$ with parametrization $\thetar \in \Thetar$ as
defined in~\cref{eqn:projneural} is pre-trained to stabilize the latent model
of unstable dynamics~\cref{eqn:rom} with respect to the equilibrium point
$(0, 0)$.
The result is a parameter $\thetar_{\ast}^{\pre}$ that is pre-trained
using the latent model only.
In Step~2, the pre-trained parameter $\thetar_{\ast}^{\pre}$ is used as
starting point for training the policy $\Krr$ defined
in~\cref{eqn:projneural} on the high-dimensional system~\cref{eqn:sys} to
obtain $\thetar_{\ast}$ with the procedure described in the previous paragraph.
The corresponding policy $K_{\thetar_{\ast}}$ is then used for stabilizing
the high-dimensional model~\cref{eqn:sys}.

Key is that the pre-trained $\thetar_{\ast}^{\pre}$ is a good starting point
for Step~2, so that fine-tuning in a multi-fidelity
sense~\cite{PehWG18, WerOP23} is cheap in terms of, e.g., high-dimensional
model queries.
In particular, the pre-trained parameter vector $\thetar_{\ast}^{\pre}$ has
to be an admissible point in the sense that it locally stabilizes the
high-dimensional system~\cref{eqn:sys} when used in~\cref{eqn:projneural}.
The \ROMProj{} method is summarized in \Cref{alg:MRLStab}.

The following theorem shows that the parameter $\thetar_{\ast}^{\pre}$ obtained
via the pre-training yields an admissible, stabilizing policy for the
high-dimensional system~\cref{eqn:sys}.

\begin{theorem} \label{thm:inittrain}
  Let $f$ in~\cref{eqn:sys} be analytic in $(\xs, \us)$ and let the parameter
  $\thetar_{\ast}^{\pre}$ be such that the policy
  $\Kr_{\thetar_{\ast}^{\pre}}$ is stabilizing for~\cref{eqn:rom} with respect
  to the zero steady state $(0, 0)$.
  Then, there exists an epsilon $\epsilon > 0$ such that
  $\Kr_{\thetar_{\ast}^{\pre}}$ in the control law~\cref{eqn:projneural} is
  locally stabilizing for~\cref{eqn:sys} for all inital values $x(0) \in
  \Xcal_{0}$ such that $\lVert x(0) - \xs \rVert \leq \epsilon$.
\end{theorem}
\begin{proof}
  From the assumptions, it follows that the Taylor series expansion of $f$
  about $(\xs, \us)$ yields a suitable approximation of the dynamics
  of~\cref{eqn:sys} via the linear system~\cref{eqn:linear}.
  From the condition that $\thetar_{\ast}^{\pre}$ is such that the
  policy $\Kr_{\thetar_{\ast}^{\pre}}$ is stabilizing
  for~\cref{eqn:rom} and the fact that the encoder and decoder
  in~\cref{eqn:unstabencode} are constructed via the basis matrix $W$
  of the left eigenspace corresponding to the unstable eigenvalues of $\Jx$,
  it follows from~\cite[Lem.~1]{WerP23a} that the policy
  $K_{\thetar_{\ast}^{\pre}}$ with~\cref{eqn:projneural} is stabilizing
  for~\cref{eqn:linear} towards the $(0, 0)$ steady state.
  Then it follows from~\cite[Prop.~1]{WerP23a} that there exists an
  $\epsilon > 0$ such that $K_{\thetar_{\ast}^{\pre}}$ is locally stabilizing
  for the equilibrium $(\xs, \us)$ in~\cref{eqn:sys}, which proofs the result
  of the theorem.
\end{proof}

Following \Cref{thm:inittrain}, we see that the initialization of
the neural network in~\cref{eqn:projneural} with $\thetar_{\ast}^{\pre}$
for the training on~\cref{eqn:sys} yields an admissible policy that
just needs to be adjusted to the high-dimensional system~\cref{eqn:sys}.


\subsection{Algorithmic description}

\begin{algorithm}[t] \small
  \SetAlgoHangIndent{1pt}
  \DontPrintSemicolon
  \caption{\resizebox{.81\linewidth}{!}{%
    [Multi-fidelity] Unstable manifold policy optimization (\algname{}).}}
  \label{alg:MRLStab}

  \KwIn{Steady state $(\xs, \us)$, queryable model $f$.}
  \KwOut{Parameters $\thetar_{\ast}$ for policy $\Ktr$
    in~\cref{eqn:projneural}.}

  Estimate the the left unstable eigenspace $\Wcal$ in~\cref{eqn:eig} using
    vector-Jacobian products with $f$ in~$(\xs, \us)$ and a Krylov
    eigenvalue solver.\;
  \label{alg:MRLStab_eig}

  \eIf{multi-fidelity}{
    Learn the low-dimensional model $(\Jxr, \Jur)$ of unstable
      dynamics~\cref{eqn:rom} using:\\
      \nonl\quad{}(a) automatic differentiation of $f$, or\\
      \nonl\quad{}(b) system identification with $\nr + \np$ forward
        evaluations of $f$.\;
    \label{alg:MRLStab_rom}

    Obtain a parameter $\thetar_{\ast}^{\pre}$ by training $\Krr$ to be a
      stabilizing policy for~\cref{eqn:rom} with respect to the
      $(0, 0)$ steady state.\;
    \label{alg:MRLStab_pre}
  }{
    Set the parameter initialization $\thetar_{\ast}^{\pre}$ to be random.\;
    \label{alg:MRLStab_pre2}
  }

  Obtain the parameter $\thetar_{\ast}$ by training $\Ktr$
    in~\cref{eqn:projneural} to be a stabilizing policy for~\cref{eqn:sys}
    with the initialization $\thetar_{\ast}^{\pre}$.\;
  \label{alg:MRLStab_train}
\end{algorithm}

\Cref{alg:MRLStab} summarizes the steps for learning stabilizing
policies on the unstable manifold with an optional pre-training step on the
latent model of unstable dynamics.
As the first step of \Cref{alg:MRLStab} in \Cref{alg:MRLStab_eig},
the basis matrix $W$ for the encoder and decoder is learned via vector-Jacobian
products with the operator $f$ as described in \Cref{sec:rlstab}.

Afterwards, if the optional pre-training is desired (determined via the
flag \emph{multi-fidelity} in \Cref{alg:MRLStab}), the latent model of
unstable dynamics~\cref{eqn:rom} is learned using the basis matrix $W$ from
the previous step and either automatic differentiation of $f$ or classical
system identification via forward evaluations of $f$.
Once the matrices in~\cref{eqn:rom} are known, this latent model is used in a
policy optimization loop to obtain a policy
$\Kr_{\thetar_{\ast}^{\pre}}$ that stabilizes~\cref{eqn:rom}
towards the $(0, 0)$ steady state.
Using~\cref{eqn:rom} for the policy optimization yields the two main
advantages compared to using directly~\cref{eqn:sys}, namely that
(i) the latent model is cheaper to evaluate since $\nr \ll \nh$, and
(ii) the dynamics of~\cref{eqn:sys} are easier to stabilize since they are
purely unstable and the training does not destabilize any additional dynamics
that have been stable before.
Also note that $\Kr_{\thetar_{\ast}^{\pre}}$ follows the theory in
\Cref{thm:inittrain} such that it is an admissible, stabilizing
policy for the high-dimensional problem already.
In the case that no pre-training is requested, the initial parametrization
is set to be random.

Lastly, in \Cref{alg:MRLStab_train} of \Cref{alg:MRLStab}, the
initialization $\thetar_{\ast}^{\pre}$ is fine tuned via policy optimization
on the high-di\-men\-sion\-al system.
In the case without pre-training, this step learns a stabilizing policy
from scratch while in the multi-fidelity version, an admissible policy
is verified on the high-dimensional system and further improved in terms
of its final performance.


\section{Numerical experiments}%
\label{sec:numerics}

We demonstrate \Proj{} and \ROMProj{} with three numerical experiments.
The reported experiments were run on compute nodes of the
\texttt{Greene} high-performance computing cluster of the New York University
equipped with 16 processing cores of the Intel Xeon Platinum 8268 24C 205W CPU
at 2.90\,GHz and 16, 32 or 48\,GB main memory.
We implemented the experiments in Python 3.9.12 running on Red Hat
Enterprise Linux release 8.4 (Ootpa).
For the reinforcement learning agents, we use the DDPG implementation from
JAXRL version~0.0.7~\cite{Kos22}.
The source codes, data and computed results of the experiments are available
at~\cite{supWer24}.


\subsection{Setup of numerical experiments}%
\label{sec:compenv}


\subsubsection{Training rewards}%
\label{sec:rewards}

A classical cost function for stabilization via optimal control is given as the
sum of the quadratic deviations of the state and the input from the controlled
steady state of interest~$(\xs, \us)$:
\begin{equation*}
  J(u(t)) = \sum\limits_{t = 0}^{\tf}
    \lVert x(t) - \xs \rVert_{Q}^{2} +
    \lVert u(t) - \us \rVert_{R}^{2},
\end{equation*}
where the weighted norms are given by
\begin{equation*}
  \lVert x(t) - \xs \rVert_{Q}^{2} = (x(t) - \xs )^{\trans} Q ( x(t) - \xs )
  \quad\text{and}\quad
  \lVert u(t) - \us \rVert_{R}^{2} = (u(t) - \us)^{\trans} R (u(t) - \us),
\end{equation*}
for some symmetric positive semi-definite matrices
$Q \in \R^{\nh \times \nh}$ and $R \in \R^{\np \times \np}$; see, for
example,~\cite{Loc01}.
This inspired the use of quadratic reward functions in control-oriented
reinforcement learning
tasks~\cite{LiTZetal22, MalPBetal19, MohZSetal22, Goe19, GraES21}.
In our setup, we employ the reward function
\begin{equation} \label{eqn:reward}
  r(x(t), u(t)) = -\sqrt{\lVert x(t) - \xs \rVert_{Q}^{2} +
    \lVert u(t) - \us \rVert_{R}^{2}},
\end{equation}
with the weighting terms to be $Q = I_{N}$ and $R = \lambda_{\mathrm{u}} I_{p}$,
and the regularization constant $\lambda_{\mathrm{u}} \geq 0$.
Additionally, early termination of episodes at a time $\ta < \tf$ due to
instabilities, where $\tf$ denotes the episode length, is penalized by
approximations of the remaining accumulated rewards as
\begin{equation*}
  r(x(\ta), u(\ta)) = -\sqrt{ (\tf - \ta)
    \lVert x(t) - \xs \rVert_{Q}^{2} +
    \lVert u(t) - \us \rVert_{R}^{2}},
\end{equation*}
to encourage the optimization towards policies with full episode length.
In the comparison of the different approaches, we consider a normalized variant
of the accumulated rewards~\cref{eqn:Prelim:RLObj} per episode given by
\begin{equation} \label{eqn:sumrew}
  \Rn = \frac{\sum\limits_{j = 0}^{\tf} r(x(t), u(t))}%
    {\sqrt{(\nc + \lambda_{\mathrm{u}}) \tf}},
\end{equation}
with episode length $\tf$, dimension of the observable $\nc$ and
the regularization constant $\lambda_{\mathrm{u}}$ from~\cref{eqn:reward}.
Note that the state-space dimension $\nc$ in~\cref{eqn:sumrew} is either
$\nh$ for the direct approach or $\nr$ in the case of the new
approaches on the unstable manifold.
We observed that the accumulated rewards range between several orders of
magnitude, therefore we use in the comparison the logarithmic mean
\begin{equation*}
  \logmean(\Rn) = -10^{\displaystyle \frac{1}{q} \sum\limits_{j = 1}^{q}
    \log_{10} \lvert \Rn^{j} \rvert},
\end{equation*}
with $q \in \N$, the number of considered elements.


\subsubsection{Governing equations of physical models}%
\label{app:pdes}

The examples considered in the following include a
nonlinear reaction-diffusion problem modeled by the Allen-Cahn
equation~\cite{AllC75, ChaI07},
a chemical reaction inside a tubular reactor~\cite{HeiP81, KraW19} and
the behavior of crystal clusters modeled via the Toda
lattice~\cite{Tod67, Wer21}.
The corresponding equations describing the physical problems and details on
the space-time discretizations can be found below.
The governing equations are numerically solved with implementations using
JAX~\cite{BraFHetal18} to enable automatic differentiation with respect to
states and controls.
The implementations of the models can be founds in the accompanying code
package~\cite{supWer24}.


\paragraph{Allan-Cahn equation.}%
This example is a nonlinear re\-ac\-tion-diffusion process described by the
one-dimensional Allen-Cahn equation~\cite{AllC75}.
The particular instance of the equation used here is sometimes also referred to
as Chafee-Infante equation~\cite{ChaI07}.
The equation with homogeneous Dirichlet boundary conditions is given by
\begin{subequations} \label{eqn:allencahn}
\begin{align}
  \partial_{\mathrm{t}} \nu(t, \zeta) - \kappa \Delta \nu(t, \zeta) +
    \alpha_{1}  \nu(t, \zeta)^{3} - \alpha_{2} \nu(t, \zeta) + u(t) & = 0,
    && \text{for}~\zeta \in \Omega,\\
  \nu(t, \zeta) & = 0, && \text{for}~\zeta \in \partial \Omega,
\end{align}
\end{subequations}
with the diffusion parameter $\kappa \in \R$, the two reaction parameters
$\alpha_{1}, \alpha_{2} \in \R$ and the distributed controls~$u(t)$.
For our experiments, the diffusion parameter has been chosen as
$\kappa = 0.2$ and the reaction parameters as
$\alpha_{1} = 2.5$ and $\alpha_{2} = 0$.
The spatial coordinates of~\cref{eqn:allencahn} in $\Omega = (0, 1)$ are
discretized using a finite difference scheme such that $\nh = 1\,000$ and
$\np = 1$.
The system is discretized in time via the implicit-explicit (IMEX) Euler scheme,
with sampling time $\tau = 0.01$, where the linear state evolution is
considered implicitly and the nonlinear term and the controls explicitly.
Note that the terms of the discrete-time system are not computed explicitly but
implemented via function evaluations.
An output operator $C$ is chosen to generate output channels via
$y(t) = C x(t)$ showing as first entry the accumulated state variables for
the first third of the spatial domain and the accumulated rest in the second
channel.
The steady state of interest is computed using the Newton iteration for the
constant steady state control $\us = 1$.
The system has one ($\nr = 1$) unstable mode in $(\xs, \us)$.


\paragraph{Tubular reactor model.}%
The tubular reactor model describes a combustion-based reaction inside a tube
that transforms reactants into desired products~\cite{HeiP81, KraW19}.
The process is described by the one-dimensional coupled partial differential
equations
\begin{subequations} \label{eqn:tubularreactor}
\begin{align}
  \partial_{\mathrm{t}} \psi(t, \zeta)
    & = \frac{1}{\mathrm{Pe}} \Delta \psi(t, \zeta)
    - \partial_{\zeta} \psi(t, \zeta)
    - \mathcal{D} \psi(t, \zeta)
    \operatorname{e}^{\gamma - \frac{\gamma}{\nu(t, \zeta)}},
    && \text{for}~\zeta \in \Omega, \\ \nonumber
  \partial_{\mathrm{t}} \nu(t, \zeta)
    & = \frac{1}{\mathrm{Pe}} \Delta \nu(t, \zeta)
    - \partial_{\zeta} \nu(t, \zeta)
    - \beta(\nu(t, \zeta) - \nu_{\mathrm{ref}} u_{2}(t)) \\
  & \quad{}+{}
    \mathcal{B} \mathcal{D} \psi(t, \zeta)
    \operatorname{e}^{\gamma - \frac{\gamma}{\nu(t, \zeta)}},
    && \text{for}~\zeta \in \Omega, \\
  0 & = \partial_{\zeta} \psi(t, \zeta)
    - \mathrm{Pe} (\psi(t, \zeta) - u_{1}(t)),
    && \text{for}~\zeta \in \partial \Omega_{1}, \\
  0 & = \partial_{\zeta} \nu(t, \zeta)
    - \mathrm{Pe} (\nu(t, \zeta) - u_{2}(t)),
    && \text{for}~\zeta \in \partial \Omega_{1}, \\
  0 & = \partial_{\zeta} \psi(t, \zeta),
    && \text{for}~\zeta \in \partial \Omega_{2}, \\
  0 & = \partial_{\zeta} \nu(t, \zeta),
    && \text{for}~\zeta \in \partial \Omega_{2},
\end{align}
\end{subequations}
where $\psi(t, \zeta)$ is the concentration of the reactant and $\nu(t, \zeta)$
the temperature of the tube.
The spatial domain is chosen as $\Omega = (0, 1)$, with
$\partial \Omega_{1} = 0$ and $\partial \Omega_{2} = 1$.
The parameters in~\cref{eqn:tubularreactor} are chosen as in~\cite{KraW19} with
$\mathrm{Pe} = 5$,
$\mathcal{D} = 0.167$,
$\gamma = 25$,
$\beta = 2.5$,
$\nu_{\mathrm{ref}} = 1$ and
$\mathcal{B} = 0.5$.
The controls in this example allow the change of the reactant concentration at
the right end of the tube as well as steering the temperature of the complete
tube.
The system is discretized in space via finite differences and in time using the
IMEX Euler scheme, with the sampling time $\tau = 0.01$,
where the linear state contributions are considered implicitly and the
nonlinear part and inputs explicitly.
If not stated otherwise, the state dimension of this example is $\nh = 998$ with
$\np = 2$ control inputs.
For the computations in \Cref{fig:training_times_full}, finer
discretizations have been used, too.
Note that the terms of the discrete-time system are not computed explicitly but
implemented via function evaluations.
An output operator $C$ is chosen to generate two measurements of the form
$y(t) = C x(t)$ giving the concentration and temperature at the left end of
the tube.
The steady state of interest is computed using the Newton iteration for the
steady state control $\us = 1$.
For any discretization size discussed here, the system has two ($\nr = 2$)
unstable modes in $(\xs, \us)$ leading under disturbances to limit cycle
oscillations of the state.


\paragraph{Toda lattice model.}%
The Toda lattice model describes the vibrational behavior of a one-dimensional
crystal structure~\cite{Tod67, Wer21}.
We consider here the process of crystallization, in which particles have formed
clusters that repel each other if disturbances occur.
The ordinary differential equations describing the process in parametrized form
were derived in~\cite{Wer21}.
A crystal with $\ell$ particles is given by
\begin{subequations}  \label{eqn:todalattice}
\begin{align}
  m_{1} \ddot{q}_{1}(t) + \gamma_{1} \dot{q}_{1}(t) +
    \operatorname{e}^{k_{1}(q_{1}(t) - q_{2}(t))} - 1 & = u_{1}(t), \\
  m_{j} \ddot{q}_{j}(t) + \gamma_{j} \dot{q}_{j}(t) +
    \operatorname{e}^{k_{j}(q_{j}(t) - q_{j + 1}(t))} -
    \operatorname{e}^{k_{j - 1}(q_{j - 1}(t) - q_{j}(t))} & = u_{j}(t), \\
  m_{\ell} \ddot{q}_{\ell}(t) + \gamma_{\ell} \dot{q}_{\ell}(t) +
    \operatorname{e}^{k_{\ell} q_{\ell}(t)} -
    \operatorname{e}^{k_{\ell - 1}(q_{\ell - 1}(t) - q_{\ell}(t))}
    & = u_{\ell}(t),
\end{align}
\end{subequations}
for $j = 2, \ldots, \ell - 1$.
Therein are $m_{i} > 0$, $\gamma_{i} > 0$ and $k_{i} \in \R$ the mass, damping
and particle forces of the $i$-th particle.
The states in~\cref{eqn:todalattice} are $q_{i}(t)$ as the displacement of
the $i$-th particle such that $\dot{q}_{i}(t)$ is the corresponding momentum.
The system~\cref{eqn:todalattice} is described by second-order differential
equations.
For the design of policies, the system is considered in first-order form
using its phase state
\begin{equation*}
  x(t) = \begin{bmatrix} q(t) \\ \dot{q}(t) \end{bmatrix} \in \R^{2 \ell},
\end{equation*}
where $q(t)$ and $\dot{q}(t)$ are the concatenated displacement and momentum
vectors of all particles.
In our experiments, we have chosen $\ell = 500$ particles such that the
considered first-order system has the order $\nh = 1\,000$.
The parameters in~\cref{eqn:todalattice} are chosen to model three clusters of
sizes $(150, 250, 100)$ with different particles of the form:
\begingroup
\allowdisplaybreaks
\begin{align*}
  m & = \mathds{1}_{100}^{\trans} \otimes \begin{bmatrix} 2 & 1 & 3 & 5 & 4
    \end{bmatrix}, \\
  \gamma_{j} & =
    \begin{cases}
      0.1 & \text{for}~j < 150, \\
      0.1 & \text{for}~j = 150, \\
      0.15 & \text{for}~150 < j < 400,\\
      0.1 & \text{for}~j = 400,\\
      0.5 & \text{for}~400 < j \leq 500,
    \end{cases}\\
  k_{j} & =
    \begin{cases}
      2 & \text{for}~j < 150, \\
      -1 & \text{for}~j = 150, \\
      5 & \text{for}~150 < j < 400,\\
      -2 & \text{for}~j = 400,\\
      1 & \text{for}~400 < j \leq 500.
    \end{cases}
\end{align*}
\endgroup
The negative forces for the particles at the end of the clusters result in
repulsion between the clusters.
The controls allow to influence the displacement of all particles inside
a cluster such that $\np = 3$.
The system is discretized in time with the same IMEX Euler scheme that is used
for the other examples such that the linear state evolution is handled
implicitly and the nonlinearity and inputs explicitly, with sampling time
$\tau = 0.1$.
Note that the terms of the discrete-time system are not computed explicitly but
implemented in function evaluations.
An output operator $C$ is chosen to generate three measurements via
$y(t) = C x(t)$ giving the mean velocity of all particles in a cluster.
The zero steady state $(0, 0)$ is considered here for stabilization.
The resulting system has two ($\nr = 2$) unstable modes in $(\xs, \us)$, which
model the system behavior that particle clusters drift indefinitely apart under
small disturbances.


\subsubsection{Hyperparameters and learning architectures}

We test three variants of the introduced approach:
\Proj{} and \ROMProj{} as described in \Cref{alg:MRLStab}, and the
policy given from the initial parametrization for \ROMProj{}, which
is learned only on the manifold approximation~\cref{eqn:rom}, further denoted
by \ROM{}.
These are compared to directly learning a stabilizing policy on the
high-dimensional system states denoted by \FOM{}.
Further details concerning the setup as well as computational
parameters and used neural network architectures can be found below.


\paragraph{Neural network architectures.}%
We implemented DDPG reinforcement learning agents using neural networks for the
actors, which approximate the desired policies, and critics, which are needed in
the optimization procedure and correspond to the value function;
see, for example,~\cite{SutB18}.
All actors in the experiments are parametrized by neural networks with two
hidden layers with \texttt{ReLU} activation
functions.
The final layer of the networks however uses either \texttt{tanh},
\texttt{identity} or \texttt{ELU}, with parameter $\alpha = 1$, as activation
function.
Note that the use of \texttt{tanh} for control-oriented neural networks
is a classical choice~\cite{LilHPetal15}, which is one motivation for the use
of policy parametrizations~\cref{eqn:projneural} that are centered around
the steady state control $\us$.
For the critics, the standard architecture from JAXRL~\cite{Kos22} has been
used.
We also experimented with networks with different widths, namely networks
with either $(20, 10)$ neurons or $(400, 300)$ neurons in the hidden layers.


\begin{table*}[t]
  \centering
  \caption{Sets of hyperparameters used in the experiments.
    The first three parameters are not varied and fixed for all experiments
    and all examples.
    For the rest, experiments have been performed for each possible parameter
    combination.
    For the regularization parameter $\lambda_{\mathrm{u}}$, we implemented two
    sets where the first one is used for Allen-Cahn and the tubular reactor
    examples and the second one for the Toda lattice example.}
  \label{tab:param}
  \vspace{.5\baselineskip}

  \begin{tabular}{ll}
    \toprule
    Maximum number of training steps & $10^{5}$ \\
    Offline phase steps before training & $256$ \\
    Noise on initial condition & $0$ \\
    \midrule
    Learning rate for actor & $\{ 0.005, 0.001, 0.0005, 0.0001 \}$ \\
    Learning rate for critic & $\{ 0.005, 0.001, 0.0005, 0.0001 \}$ \\
    Exploration noise & $\{ 0.001, 0 \}$ \\
    Episode length $\tf$ & $\{ 100, 200, 500 \}$ \\
    Regularization parameter $\lambda_{\mathrm{u}}$ & $\{ 1000, 0.001, 0 \}$ or
      $\{ 10^{6}, 10^{5}, 10^{4} \}$ \\
    \bottomrule
  \end{tabular}
\end{table*}

\paragraph{Choices of hyperparameters.}%
We performed experiments with fixed random seeds for the initialization of the
neural networks and varied the hyperparameters of the learning algorithms.
\Cref{tab:param} shows an overview about the considered parameters
in the experiments.
For each possible selection from the parameter sets, an experiment has been
performed for each example.
The last parameter $\lambda_{\mathrm{u}}$ has two sets, where the first
one is used in the Allen-Cahn and tubular reactor examples and the
second one in the Toda lattice example.
Together with the variations in the neural network architectures,
$1\,728$ experiments have been run per example and for each of the
four presented methods.
Some of the hyperparameters in \Cref{tab:param} have been fixed for all
experiments.
The maximum number of training steps is fixed for comparability between
the experiments, and for variations in the number of offline steps, no
noticeable difference in the results has been observed.
We apply some noise to the initial condition when resetting a training episode
to allow the agent to explore and be trained in different scenarios.
However, we have observed that the general inability of neural networks
to approximate constants, in particular the zero, has a similar effect
as initial noise, especially in the early stages of the training.
Additional experiments have been performed to investigate the influence of the
state-space dimension $\nh$ on the training time as shown in
\Cref{fig:training_times_full}.
For these experiments, we reduced the number of considered hyperparameters by
allowing for the learning rates of actor and critic only $\{0.001, 0.0001\}$.
This enabled us to fully run the $10^{5}$ training steps also for the
\FOM{} approach.


\subsection{Numerical results}%
\label{sec:results}

In the following, we present the numerical results obtained for the different
experiments described above.
Further results beyond those presented here can be found in the
accompanying code package~\cite{supWer24}.


\subsubsection{Training overview}%
\label{sec:overview}

\begin{figure*}
  \centering
  \begin{subfigure}[b]{.495\linewidth}
    \centering
  \tikzexternalenable%
  \tikzsetnextfilename{performance_reactiondiffusion_1}%
  \begin{tikzpicture}[
  fill between/on layer = {main},
  font                  = \plotfontsize,
  text                  = black
]
  \pgfplotstableread{graphics/data/performance_reactiondiffusion_net=20_10_FOM.dat}\tableFOM
  \pgfplotstableread{graphics/data/performance_reactiondiffusion_net=20_10_Proj.dat}\tableProj
  \pgfplotstableread{graphics/data/performance_reactiondiffusion_net=20_10_ROM.dat}\tableROM
  \pgfplotstableread{graphics/data/performance_reactiondiffusion_net=20_10_ROM_Proj.dat}\tableROMProj
  
  \begin{semilogyaxis}[%
    scale only axis,
    width           = .775\linewidth,
    height          = .325\linewidth,
    xmin            = 0,
    xmax            = 0.059,
    ymin            = 1e-5,
    ymax            = 2e+3,
    y dir           = reverse,
    xminorticks     = false,
    yminorticks     = false,
    yticklabels     = {$-10^{-5}$, $-10^{-1}$, $-10^{3}$},
    scaled x ticks  = false,
    xticklabels     = {, $0$, $0.01$, $0.02$, $0.03$, $0.04$, $0.05$},
    xlabel          = {training time (h)},
    ylabel          = {norm. acc. reward $\Rn$},
    ylabel style    = {yshift = -.4em},
    xlabel style    = {yshift = .4em},
    cycle list name = performancelist
  ]

    \addplot+[name path = A]
      table[x index = 0, y expr = -\thisrowno{1}] {\tableFOM};
    \addplot+[name path = B]
      table[x index = 0, y expr = -\thisrowno{2}] {\tableFOM};
    \addplot+ fill between [of = A and B];
    \addplot+ table[x index = 0, y expr = -\thisrowno{3}] {\tableFOM};

    \addplot+[name path = A]
      table[x index = 0, y expr = -\thisrowno{1}] {\tableProj};
    \addplot+[name path = B]
      table[x index = 0, y expr = -\thisrowno{2}] {\tableProj};
    \addplot+ fill between [of = A and B];
    \addplot+ table[x index = 0, y expr = -\thisrowno{3}] {\tableProj};

    \addplot+[name path = A]
      table[x index = 0, y expr = -\thisrowno{1}] {\tableROM};
    \addplot+[name path = B]
      table[x index = 0, y expr = -\thisrowno{2}] {\tableROM};
    \addplot+ fill between [of = A and B];
    \addplot+ table[x index = 0, y expr = -\thisrowno{3}] {\tableROM};

    \addplot+[name path = A]
      table[x index = 0, y expr = -\thisrowno{1}] {\tableROMProj};
    \addplot+[name path = B]
      table[x index = 0, y expr = -\thisrowno{2}] {\tableROMProj};
    \addplot+ fill between [of = A and B];
    \addplot+ table[x index = 0, y expr = -\thisrowno{3}] {\tableROMProj};

  \end{semilogyaxis}
\end{tikzpicture}%
  \tikzexternaldisable%

    \vspace{-1.25\baselineskip}
    \caption{Allen-Cahn; network $(20, 10)$}
    \label{fig:performance_1_reactiondiffusion}
  \end{subfigure}%
  \hfill%
  \begin{subfigure}[b]{.495\linewidth}
    \centering
  \tikzexternalenable%
  \tikzsetnextfilename{performance_reactiondiffusion_2}%
  \begin{tikzpicture}[
  fill between/on layer = {main},
  font                  = \plotfontsize,
  text                  = black
]
  \pgfplotstableread{graphics/data/performance_reactiondiffusion_net=400_300_FOM.dat}\tableFOM
  \pgfplotstableread{graphics/data/performance_reactiondiffusion_net=400_300_Proj.dat}\tableProj
  \pgfplotstableread{graphics/data/performance_reactiondiffusion_net=400_300_ROM.dat}\tableROM
  \pgfplotstableread{graphics/data/performance_reactiondiffusion_net=400_300_ROM_Proj.dat}\tableROMProj
  
  \begin{semilogyaxis}[%
    scale only axis,
    width           = .775\linewidth,
    height          = .325\linewidth,
    xmin            = 0,
    xmax            = 0.43,
    ymin            = 1e-4,
    ymax            = 2e+3,
    y dir           = reverse,
    xminorticks     = false,
    yminorticks     = false,
    yticklabels     = {$-10^{-4}$, $-10^{-1}$, $-10^{2}$},
    scaled x ticks  = false,
    xlabel          = {training time (h)},
    ylabel          = {norm. acc. reward $\Rn$},
    ylabel style    = {yshift = -.4em},
    xlabel style    = {yshift = .4em},
    cycle list name = performancelist
  ]

    \addplot+[name path = A]
      table[x index = 0, y expr = -\thisrowno{1}] {\tableFOM};
    \addplot+[name path = B]
      table[x index = 0, y expr = -\thisrowno{2}] {\tableFOM};
    \addplot+ fill between [of = A and B];
    \addplot+ table[x index = 0, y expr = -\thisrowno{3}] {\tableFOM};

    \addplot+[name path = A]
      table[x index = 0, y expr = -\thisrowno{1}] {\tableProj};
    \addplot+[name path = B]
      table[x index = 0, y expr = -\thisrowno{2}] {\tableProj};
    \addplot+ fill between [of = A and B];
    \addplot+ table[x index = 0, y expr = -\thisrowno{3}] {\tableProj};

    \addplot+[name path = A]
      table[x index = 0, y expr = -\thisrowno{1}] {\tableROM};
    \addplot+[name path = B]
      table[x index = 0, y expr = -\thisrowno{2}] {\tableROM};
    \addplot+ fill between [of = A and B];
    \addplot+ table[x index = 0, y expr = -\thisrowno{3}] {\tableROM};

    \addplot+[name path = A]
      table[x index = 0, y expr = -\thisrowno{1}] {\tableROMProj};
    \addplot+[name path = B]
      table[x index = 0, y expr = -\thisrowno{2}] {\tableROMProj};
    \addplot+ fill between [of = A and B];
    \addplot+ table[x index = 0, y expr = -\thisrowno{3}] {\tableROMProj};

  \end{semilogyaxis}
\end{tikzpicture}%
  \tikzexternaldisable%

    \vspace{-1.25\baselineskip}
    \caption{Allen-Cahn; network $(400, 300)$}
    \label{fig:performance_2_reactiondiffusion}
  \end{subfigure}

  \vspace{.5\baselineskip}
  \begin{subfigure}[b]{.495\linewidth}
    \centering
  \tikzexternalenable%
  \tikzsetnextfilename{performance_tubularreactor_1}%
  \begin{tikzpicture}[
  fill between/on layer = {main},
  font                  = \plotfontsize,
  text                  = black
]
  \pgfplotstableread{graphics/data/performance_tubularreactor_net=20_10_FOM.dat}\tableFOM
  \pgfplotstableread{graphics/data/performance_tubularreactor_net=20_10_Proj.dat}\tableProj
  \pgfplotstableread{graphics/data/performance_tubularreactor_net=20_10_ROM.dat}\tableROM
  \pgfplotstableread{graphics/data/performance_tubularreactor_net=20_10_ROM_Proj.dat}\tableROMProj
  
  \begin{semilogyaxis}[%
    scale only axis,
    width           = .775\linewidth,
    height          = .325\linewidth,
    xmin            = 0,
    xmax            = 0.24,
    ymin            = 1e-4,
    ymax            = 1e+5,
    y dir           = reverse,
    xminorticks     = false,
    yminorticks     = false,
    yticklabels     = {$-10^{-4}$, $-10^{1}$, $-10^{4}$},
    scaled x ticks  = false,
    xticklabels     = {, $0$, $0.05$, $0.1$, $0.15$, $0.2$},
    xlabel          = {training time (h)},
    ylabel          = {norm. acc. reward $\Rn$},
    ylabel style    = {yshift = -.4em},
    xlabel style    = {yshift = .4em},
    cycle list name = performancelist
  ]

    \addplot+[name path = A]
      table[x index = 0, y expr = -\thisrowno{1}] {\tableFOM};
    \addplot+[name path = B]
      table[x index = 0, y expr = -\thisrowno{2}] {\tableFOM};
    \addplot+ fill between [of = A and B];
    \addplot+ table[x index = 0, y expr = -\thisrowno{3}] {\tableFOM};

    \addplot+[name path = A]
      table[x index = 0, y expr = -\thisrowno{1}] {\tableProj};
    \addplot+[name path = B]
      table[x index = 0, y expr = -\thisrowno{2}] {\tableProj};
    \addplot+ fill between [of = A and B];
    \addplot+ table[x index = 0, y expr = -\thisrowno{3}] {\tableProj};

    \addplot+[name path = A]
      table[x index = 0, y expr = -\thisrowno{1}] {\tableROM};
    \addplot+[name path = B]
      table[x index = 0, y expr = -\thisrowno{2}] {\tableROM};
    \addplot+ fill between [of = A and B];
    \addplot+ table[x index = 0, y expr = -\thisrowno{3}] {\tableROM};

    \addplot+[name path = A]
      table[x index = 0, y expr = -\thisrowno{1}] {\tableROMProj};
    \addplot+[name path = B]
      table[x index = 0, y expr = -\thisrowno{2}] {\tableROMProj};
    \addplot+ fill between [of = A and B];
    \addplot+ table[x index = 0, y expr = -\thisrowno{3}] {\tableROMProj};

  \end{semilogyaxis}
\end{tikzpicture}%
  \tikzexternaldisable%

    \vspace{-1.25\baselineskip}
    \caption{tubular reactor; network $(20, 10)$}
    \label{fig:performance_1_tubularreactor}
  \end{subfigure}
  \hfill
  \begin{subfigure}[b]{.495\linewidth}
    \centering
  \tikzexternalenable%
  \tikzsetnextfilename{performance_tubularreactor_2}%
  \begin{tikzpicture}[
  fill between/on layer = {main},
  font                  = \plotfontsize,
  text                  = black
]
  \pgfplotstableread{graphics/data/performance_tubularreactor_net=400_300_FOM.dat}\tableFOM
  \pgfplotstableread{graphics/data/performance_tubularreactor_net=400_300_Proj.dat}\tableProj
  \pgfplotstableread{graphics/data/performance_tubularreactor_net=400_300_ROM.dat}\tableROM
  \pgfplotstableread{graphics/data/performance_tubularreactor_net=400_300_ROM_Proj.dat}\tableROMProj
  
  \begin{semilogyaxis}[%
    scale only axis,
    width           = .775\linewidth,
    height          = .325\linewidth,
    xmin            = 0,
    xmax            = 0.54,
    ymin            = 1e-4,
    ymax            = 1e+6,
    y dir           = reverse,
    xminorticks     = false,
    yminorticks     = false,
    yticklabels     = {$-10^{-4}$, $-10^{1}$, $-10^{6}$},
    scaled x ticks  = false,
    xticklabels     = {, $0$, $0.1$, $0.2$, $0.3$, $0.4$, $0.5$},
    xlabel          = {training time (h)},
    ylabel          = {norm. acc. reward $\Rn$},
    ylabel style    = {yshift = -.4em},
    xlabel style    = {yshift = .4em},
    cycle list name = performancelist
  ]

    \addplot+[name path = A]
      table[x index = 0, y expr = -\thisrowno{1}] {\tableFOM};
    \addplot+[name path = B]
      table[x index = 0, y expr = -\thisrowno{2}] {\tableFOM};
    \addplot+ fill between [of = A and B];
    \addplot+ table[x index = 0, y expr = -\thisrowno{3}] {\tableFOM};

    \addplot+[name path = A]
      table[x index = 0, y expr = -\thisrowno{1}] {\tableProj};
    \addplot+[name path = B]
      table[x index = 0, y expr = -\thisrowno{2}] {\tableProj};
    \addplot+ fill between [of = A and B];
    \addplot+ table[x index = 0, y expr = -\thisrowno{3}] {\tableProj};

    \addplot+[name path = A]
      table[x index = 0, y expr = -\thisrowno{1}] {\tableROM};
    \addplot+[name path = B]
      table[x index = 0, y expr = -\thisrowno{2}] {\tableROM};
    \addplot+ fill between [of = A and B];
    \addplot+ table[x index = 0, y expr = -\thisrowno{3}] {\tableROM};

    \addplot+[name path = A]
      table[x index = 0, y expr = -\thisrowno{1}] {\tableROMProj};
    \addplot+[name path = B]
      table[x index = 0, y expr = -\thisrowno{2}] {\tableROMProj};
    \addplot+ fill between [of = A and B];
    \addplot+ table[x index = 0, y expr = -\thisrowno{3}] {\tableROMProj};

  \end{semilogyaxis}
\end{tikzpicture}%
  \tikzexternaldisable%

    \vspace{-1.25\baselineskip}
    \caption{tubular reactor; network $(400, 300)$}
    \label{fig:performance_2_tubularreactor}
  \end{subfigure}

  \vspace{.5\baselineskip}
  \begin{subfigure}[b]{.495\linewidth}
    \centering
  \tikzexternalenable%
  \tikzsetnextfilename{performance_todalattice_1}%
  \begin{tikzpicture}[
  fill between/on layer = {main},
  font                  = \plotfontsize,
  text                  = black
]
  \pgfplotstableread{graphics/data/performance_todalattice_net=20_10_FOM.dat}\tableFOM
  \pgfplotstableread{graphics/data/performance_todalattice_net=20_10_Proj.dat}\tableProj
  \pgfplotstableread{graphics/data/performance_todalattice_net=20_10_ROM.dat}\tableROM
  \pgfplotstableread{graphics/data/performance_todalattice_net=20_10_ROM_Proj.dat}\tableROMProj
  
  \begin{semilogyaxis}[%
    scale only axis,
    width           = .775\linewidth,
    height          = .325\linewidth,
    xmin            = 0,
    xmax            = 0.137,
    ymin            = 1e-2,
    ymax            = 1e+6,
    y dir           = reverse,
    xminorticks     = false,
    yminorticks     = false,
    yticklabels     = {$-10^{-2}$, $-10^{2}$, $-10^{6}$},
    scaled x ticks  = false,
    xticklabels     = {, $0$, $0.02$, $0.04$, $0.06$, $0.08$, $0.1$, $0.12$},
    xlabel          = {training time (h)},
    ylabel          = {norm. acc. reward $\Rn$},
    ylabel style    = {yshift = -.4em},
    xlabel style    = {yshift = .4em},
    cycle list name = performancelist
  ]

    \addplot+[name path = A]
      table[x index = 0, y expr = -\thisrowno{1}] {\tableFOM};
    \addplot+[name path = B]
      table[x index = 0, y expr = -\thisrowno{2}] {\tableFOM};
    \addplot+ fill between [of = A and B];
    \addplot+ table[x index = 0, y expr = -\thisrowno{3}] {\tableFOM};

    \addplot+[name path = A]
      table[x index = 0, y expr = -\thisrowno{1}] {\tableProj};
    \addplot+[name path = B]
      table[x index = 0, y expr = -\thisrowno{2}] {\tableProj};
    \addplot+ fill between [of = A and B];
    \addplot+ table[x index = 0, y expr = -\thisrowno{3}] {\tableProj};

    \addplot+[name path = A]
      table[x index = 0, y expr = -\thisrowno{1}] {\tableROM};
    \addplot+[name path = B]
      table[x index = 0, y expr = -\thisrowno{2}] {\tableROM};
    \addplot+ fill between [of = A and B];
    \addplot+ table[x index = 0, y expr = -\thisrowno{3}] {\tableROM};

    \addplot+[name path = A]
      table[x index = 0, y expr = -\thisrowno{1}] {\tableROMProj};
    \addplot+[name path = B]
      table[x index = 0, y expr = -\thisrowno{2}] {\tableROMProj};
    \addplot+ fill between [of = A and B];
    \addplot+ table[x index = 0, y expr = -\thisrowno{3}] {\tableROMProj};

  \end{semilogyaxis}
\end{tikzpicture}%
  \tikzexternaldisable%

    \vspace{-1.25\baselineskip}
    \caption{Toda lattice; network $(20, 10)$}
    \label{fig:performance_1_todalattice1}
  \end{subfigure}%
  \hfill%
  \begin{subfigure}[b]{.495\linewidth}
    \centering
  \tikzexternalenable%
  \tikzsetnextfilename{performance_todalattice_2}%
  \begin{tikzpicture}[
  fill between/on layer = {main},
  font                  = \plotfontsize,
  text                  = black
]
  \pgfplotstableread{graphics/data/performance_todalattice_net=400_300_FOM.dat}\tableFOM
  \pgfplotstableread{graphics/data/performance_todalattice_net=400_300_Proj.dat}\tableProj
  \pgfplotstableread{graphics/data/performance_todalattice_net=400_300_ROM.dat}\tableROM
  \pgfplotstableread{graphics/data/performance_todalattice_net=400_300_ROM_Proj.dat}\tableROMProj
  
  \begin{semilogyaxis}[%
    scale only axis,
    width           = .775\linewidth,
    height          = .325\linewidth,
    xmin            = 0,
    xmax            = 0.42,
    ymin            = 1e-3,
    ymax            = 1e+6,
    y dir           = reverse,
    xminorticks     = false,
    yminorticks     = false,
    yticklabels     = {$-10^{-3}$, $-10^{1}$, $-10^{5}$},
    scaled x ticks  = false,
    xlabel          = {training time (h)},
    ylabel          = {norm. acc. reward $\Rn$},
    ylabel style    = {yshift = -.4em},
    xlabel style    = {yshift = .4em},
    cycle list name = performancelist
  ]

    \addplot+[name path = A]
      table[x index = 0, y expr = -\thisrowno{1}] {\tableFOM};
    \addplot+[name path = B]
      table[x index = 0, y expr = -\thisrowno{2}] {\tableFOM};
    \addplot+ fill between [of = A and B];
    \addplot+ table[x index = 0, y expr = -\thisrowno{3}] {\tableFOM};

    \addplot+[name path = A]
      table[x index = 0, y expr = -\thisrowno{1}] {\tableProj};
    \addplot+[name path = B]
      table[x index = 0, y expr = -\thisrowno{2}] {\tableProj};
    \addplot+ fill between [of = A and B];
    \addplot+ table[x index = 0, y expr = -\thisrowno{3}] {\tableProj};

    \addplot+[name path = A]
      table[x index = 0, y expr = -\thisrowno{1}] {\tableROM};
    \addplot+[name path = B]
      table[x index = 0, y expr = -\thisrowno{2}] {\tableROM};
    \addplot+ fill between [of = A and B];
    \addplot+ table[x index = 0, y expr = -\thisrowno{3}] {\tableROM};

    \addplot+[name path = A]
      table[x index = 0, y expr = -\thisrowno{1}] {\tableROMProj};
    \addplot+[name path = B]
      table[x index = 0, y expr = -\thisrowno{2}] {\tableROMProj};
    \addplot+ fill between [of = A and B];
    \addplot+ table[x index = 0, y expr = -\thisrowno{3}] {\tableROMProj};

  \end{semilogyaxis}
\end{tikzpicture}%
  \tikzexternaldisable%

    \vspace{-1.25\baselineskip}
    \caption{Toda lattice; network $(400, 300)$}
    \label{fig:performance_2_todalattice1}
  \end{subfigure}%

  \vspace{.5\baselineskip}
  \begin{subfigure}[b]{.495\linewidth}
    \centering
  \tikzexternalenable%
  \tikzsetnextfilename{performance_todalattice_zoom_1}%
  \begin{tikzpicture}[
  fill between/on layer = {main},
  font                  = \plotfontsize,
  text                  = black
]
  \pgfplotstableread{graphics/data/performance_todalattice_net=20_10_FOM.dat}\tableFOM
  \pgfplotstableread{graphics/data/performance_todalattice_net=20_10_Proj.dat}\tableProj
  \pgfplotstableread{graphics/data/performance_todalattice_net=20_10_ROM.dat}\tableROM
  \pgfplotstableread{graphics/data/performance_todalattice_net=20_10_ROM_Proj.dat}\tableROMProj
  
  \begin{semilogyaxis}[%
    scale only axis,
    width           = .775\linewidth,
    height          = .325\linewidth,
    xmin            = 0,
    xmax            = 0.022,
    ymin            = 1e-2,
    ymax            = 1e+2,
    y dir           = reverse,
    xminorticks     = false,
    yminorticks     = false,
    yticklabels     = {, $-10^{-1}$, $-10^{1}$},
    scaled x ticks  = false,
    xticklabels     = {, $0$, $0.005$, $0.01$, $0.015$, $0.02$},
    xlabel          = {training time (h)},
    ylabel          = {norm. acc. reward $\Rn$},
    ylabel style    = {yshift = -.4em},
    xlabel style    = {yshift = .4em},
    cycle list name = performancelist
  ]

    \addplot+[name path = A]
      table[x index = 0, y expr = -\thisrowno{1}] {\tableFOM};
    \addplot+[name path = B]
      table[x index = 0, y expr = -\thisrowno{2}] {\tableFOM};
    \addplot+ fill between [of = A and B];
    \addplot+ table[x index = 0, y expr = -\thisrowno{3}] {\tableFOM};

    \addplot+[name path = A]
      table[x index = 0, y expr = -\thisrowno{1}] {\tableProj};
    \addplot+[name path = B]
      table[x index = 0, y expr = -\thisrowno{2}] {\tableProj};
    \addplot+ fill between [of = A and B];
    \addplot+ table[x index = 0, y expr = -\thisrowno{3}] {\tableProj};

    \addplot+[name path = A]
      table[x index = 0, y expr = -\thisrowno{1}] {\tableROM};
    \addplot+[name path = B]
      table[x index = 0, y expr = -\thisrowno{2}] {\tableROM};
    \addplot+ fill between [of = A and B];
    \addplot+ table[x index = 0, y expr = -\thisrowno{3}] {\tableROM};

    \addplot+[name path = A]
      table[x index = 0, y expr = -\thisrowno{1}] {\tableROMProj};
    \addplot+[name path = B]
      table[x index = 0, y expr = -\thisrowno{2}] {\tableROMProj};
    \addplot+ fill between [of = A and B];
    \addplot+ table[x index = 0, y expr = -\thisrowno{3}] {\tableROMProj};

  \end{semilogyaxis}
\end{tikzpicture}%
  \tikzexternaldisable%

    \vspace{-1.25\baselineskip}
    \caption{Toda lattice (zoom in); network $(20, 10)$}
    \label{fig:performance_1_todalattice2}
  \end{subfigure}
  \hfill
  \begin{subfigure}[b]{.495\linewidth}
    \centering
  \tikzexternalenable%
  \tikzsetnextfilename{performance_todalattice_zoom_2}%
  \begin{tikzpicture}[
  fill between/on layer = {main},
  font                  = \plotfontsize,
  text                  = black
]
  \pgfplotstableread{graphics/data/performance_todalattice_net=400_300_FOM.dat}\tableFOM
  \pgfplotstableread{graphics/data/performance_todalattice_net=400_300_Proj.dat}\tableProj
  \pgfplotstableread{graphics/data/performance_todalattice_net=400_300_ROM.dat}\tableROM
  \pgfplotstableread{graphics/data/performance_todalattice_net=400_300_ROM_Proj.dat}\tableROMProj
  
  \begin{semilogyaxis}[%
    scale only axis,
    width           = .775\linewidth,
    height          = .325\linewidth,
    xmin            = 0,
    xmax            = 0.18,
    ymin            = 4e-3,
    ymax            = 4e+0,
    y dir           = reverse,
    xminorticks     = false,
    yminorticks     = false,
    yticklabels     = {, $-10^{-2}$, $-10^{-1}$, $-10^{0}$},
    scaled x ticks  = false,
    xticklabels     = {, $0$, $0.05$, $0.1$, $0.15$},
    xlabel          = {training time (h)},
    ylabel          = {norm. acc. reward $\Rn$},
    ylabel style    = {yshift = -.4em},
    xlabel style    = {yshift = .4em},
    cycle list name = performancelist
  ]

    \addplot+[name path = A]
      table[x index = 0, y expr = -\thisrowno{1}] {\tableFOM};
    \addplot+[name path = B]
      table[x index = 0, y expr = -\thisrowno{2}] {\tableFOM};
    \addplot+ fill between [of = A and B];
    \addplot+ table[x index = 0, y expr = -\thisrowno{3}] {\tableFOM};

    \addplot+[name path = A]
      table[x index = 0, y expr = -\thisrowno{1}] {\tableProj};
    \addplot+[name path = B]
      table[x index = 0, y expr = -\thisrowno{2}] {\tableProj};
    \addplot+ fill between [of = A and B];
    \addplot+ table[x index = 0, y expr = -\thisrowno{3}] {\tableProj};

    \addplot+[name path = A]
      table[x index = 0, y expr = -\thisrowno{1}] {\tableROM};
    \addplot+[name path = B]
      table[x index = 0, y expr = -\thisrowno{2}] {\tableROM};
    \addplot+ fill between [of = A and B];
    \addplot+ table[x index = 0, y expr = -\thisrowno{3}] {\tableROM};

    \addplot+[name path = A]
      table[x index = 0, y expr = -\thisrowno{1}] {\tableROMProj};
    \addplot+[name path = B]
      table[x index = 0, y expr = -\thisrowno{2}] {\tableROMProj};
    \addplot+ fill between [of = A and B];
    \addplot+ table[x index = 0, y expr = -\thisrowno{3}] {\tableROMProj};

  \end{semilogyaxis}
\end{tikzpicture}%
  \tikzexternaldisable%

    \vspace{-1.25\baselineskip}
    \caption{Toda lattice (zoom in); network $(400, 300)$}
    \label{fig:performance_2_todalattice2}
  \end{subfigure}

  \vspace{.5\baselineskip}
  \tikzexternalenable%
  \tikzsetnextfilename{performance_legend}%
  \begin{tikzpicture}[font = \legendfontsize, text = black]
  \begin{axis}[%
    hide axis,
    width  = 1mm,
    height = 1mm,
    scale only axis,
    xmin = 0,
    xmax = 1,
    ymin = 0,
    ymax = 1,
    legend columns = -1, 
    legend style   = {
      at     = {(0,0)},
      anchor = center,
      /tikz/every even column/.append style = {column sep = 0.4cm}},
    legend cell align  = {left},
    clip mode          = individual,
    cycle list name    = performancelist
  ]

    \pgfplotsset{cycle list shift = 3}
    \addplot+ coordinates{(0, 0)};
    \addlegendentry{\FOM{}}

    \pgfplotsset{cycle list shift = 6}
    \addplot+ coordinates{(0, 0)};
    \addlegendentry{\Proj{}}

    \pgfplotsset{cycle list shift = 9}
    \addplot+ coordinates{(0, 0)};
    \addlegendentry{\ROM{}}

    \pgfplotsset{cycle list shift = 12}
    \addplot+ coordinates{(0, 0)};
    \addlegendentry{\ROMProj{}}
  \end{axis}
\end{tikzpicture}%
  \tikzexternaldisable%

  \caption{Comparison of normalized accumulated rewards:
    Plots (a)--(d) show that the approach \ROMProj{} that uses the
    high-dimensional system and the latent model together achieves higher
    rewards than \ROM{} that uses the latent model alone.
    For the Toda lattice example with results shown in (e)--(g), \ROMProj{}
    achieves similar rewards as \FOM{}.
    Note that in the Toda lattice example, \ROM{} achieves the highest rewards
    as the latent model of the unstable dynamics hides many of the strongly
    nonlinear dynamics that affect the stabilization.}
  \label{fig:performance}
\end{figure*}
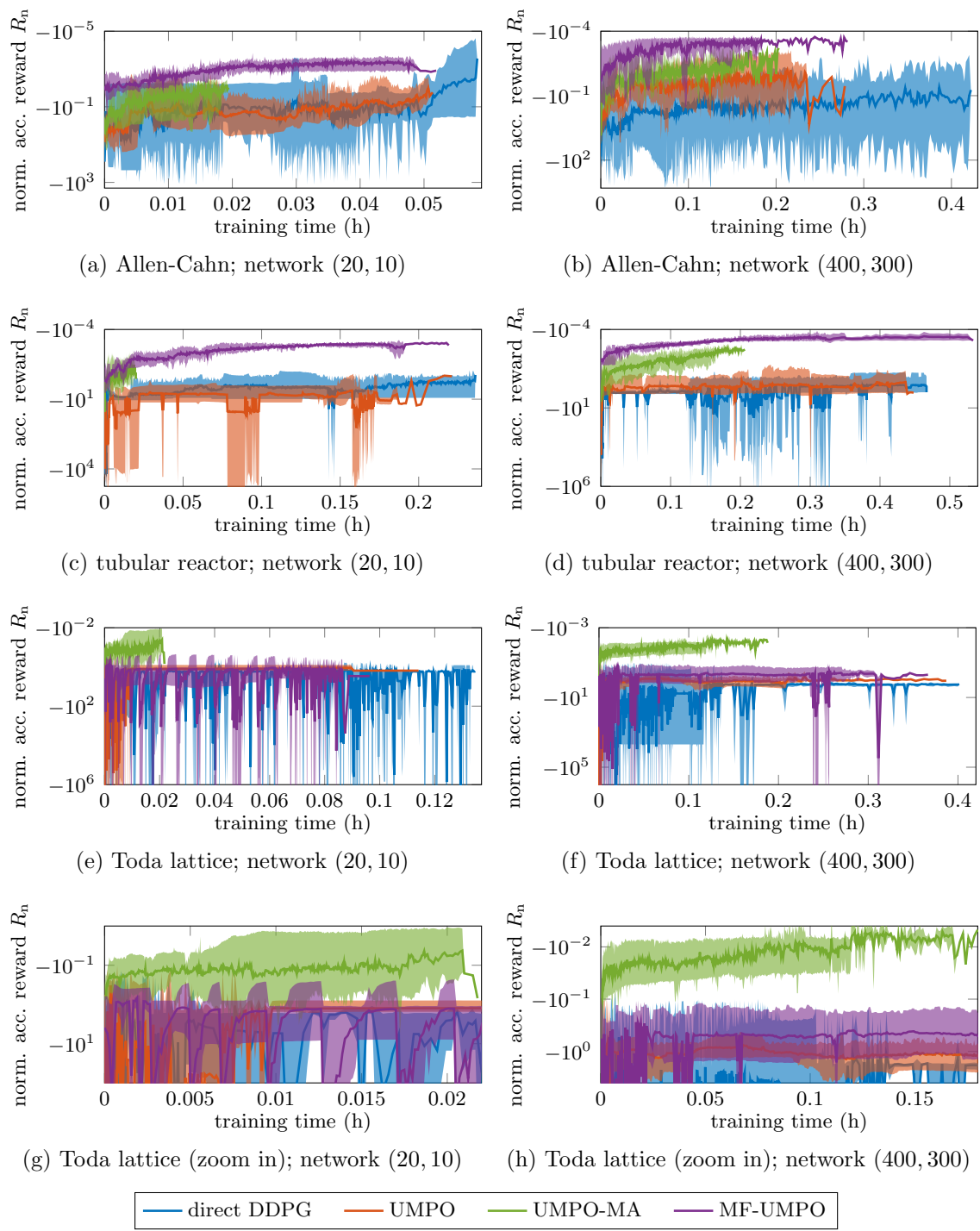

The detailed training behavior for selected policies is shown in
\Cref{fig:performance}.
Since a significant amount of hyperparameter setups did not result in
stabilizing policies in the test simulation, in particular for the classical
\FOM{} approach, we filtered the results to show in
\Cref{fig:performance} only the five best policies that
stabilize the systems with smallest distance to the desired steady state.
The training time of \ROM{} is not included in the plots of \ROMProj{} in these
figures.
Additionally, the training runs of \FOM{} have been restricted in terms of
computation time to be at most as long as the overall slowest runs of
\Proj{} for the respective neural network sizes.
Comparing the results for the Allen-Cahn equation shows that in both cases the
new approaches reach faster higher rewards than \FOM{}.
For smaller networks such as $(20,10)$, the rewards of \Proj{} behave
similar to \FOM{}.
\ROM{} and \ROMProj{} provide higher rewards faster during their training
phase, but \FOM{} is able to catch up to these after $0.05$\,h.
In contrast, for larger networks such as $(400, 200)$, the new approaches
\Proj{}, \ROM{} and \ROMProj{} clearly outperform the classical \FOM{} method
in terms of reward versus training time.
While as in \Cref{fig:performance_1_reactiondiffusion} it would be possible
for \FOM{} to improve its rewards for longer training times, it already runs
twice as long as the other approaches implying that significantly more
resources are needed for \FOM{} to give higher rewards here.
For the tubular reactor, the results in
\Cref{fig:performance_1_tubularreactor,fig:performance_2_tubularreactor}
look very similar to each other, with the rewards of \Proj{} behaving
similar to \FOM{} but with shorter training times.
\ROM{} gives higher rewards than \ROM{} and \Proj{}, which are further
improved by \ROMProj{}.
Lastly, we have the Toda lattice example in
\Cref{fig:performance_1_todalattice1,fig:performance_1_todalattice2} as well
as in \Cref{fig:performance_2_todalattice1,fig:performance_2_todalattice2}.
The strong nonlinearity is a challenge for reinforcement learning methods that
are trained based on querying the high-dimensional
system~\cref{eqn:sys}.
This can be also seen here as \FOM{}, \Proj{} and \ROMProj{} provide similar
reward behaviors, while \ROM{} has the best performance.
However, we can see in
\Cref{fig:performance_1_todalattice2,fig:performance_2_todalattice2} that
\ROMProj{} still yields consistently better rewards than \FOM{}.
For this example, \Proj{} and \ROMProj{} achieve improvements by a factor
of $2$ in terms of the mean rewards over \FOM{}.

\begin{figure*}[t]
  \centering
  \tikzexternalenable%
  \tikzsetnextfilename{queries_full}%
  \begin{tikzpicture}[font = \plotfontsize, text = black]
  \begin{axis}[
    scale only axis,
    width             = .775\linewidth,
    height            = .19\linewidth,
    ybar,
    bar width         = 6pt,
    ymin              = 0,
    ymax              = 1e+5,
    ylabel            = {mean \# system queries},
    ymajorgrids,
    ylabel style      = {yshift = -.4em},
    xtick             = data,
    symbolic x coords = {test11, test12, test21, test22, test31, test32},
    xticklabel style  = {yshift = .4em},
    point meta        = explicit symbolic,
    xticklabels       = {(20{,} 10), (400{,} 300), (20{,} 10), (400{,} 300),
                          (20{,} 10), (400{,} 300)},
    legend columns    = -1,
    legend cell align = {left},
    legend style      = {
      at                                    = {(.46,-0.375)},
      anchor                                = north,
      /tikz/every even column/.append style = {column sep = 0.25cm}},
    cycle list name   = barlist
  ]

    \addplot+[ybar] coordinates{
      (test11, 15060.40)
      (test12, 38396.20)
      (test21, 6458.00)
      (test22, 27102.20)
      (test31, 11303.80)
      (test32, 36618.60)
    };
    \addlegendentry{\legendfontsize \FOM{}}

    \addplot+[ybar] coordinates{
      (test11, 39388.60)
      (test12, 85521.60)
      (test21, 10586.60)
      (test22, 47937.00)
      (test31, 23839.80)
      (test32, 77800.00)
    };
    \addlegendentry{\legendfontsize \Proj{}}

    \addplot+[ybar] coordinates{
      (test11, 99942.60)
      (test12, 99586.60)
      (test21, 99938.00)
      (test22, 99057.40)
      (test31, 97319.00)
      (test32, 99807.00)
    };
    \addlegendentry{\legendfontsize \ROM{}}

    \addplot+[ybar] coordinates{
      (test11, 42740.00)
      (test12, 94200.00)
      (test21, 10680.00)
      (test22, 50860.00)
      (test31, 23044.60)
      (test32, 64743.40)
    };
    \addlegendentry{\legendfontsize \ROMProj{}}
  \end{axis}

  \node at (1.95, -.7) [anchor = center] {\textbf{Allen-Cahn}};
  \node at (5.8, -.7) [anchor = center] {\textbf{Tubular reactor}};
  \node at (9.7, -.7) [anchor = center] {\textbf{Toda lattice}};
\end{tikzpicture}%
  \tikzexternaldisable%

  \caption{Comparison of the mean amount of system queries after
    $0.02$\,h and $0.2$\,h training time for small and large
    neural networks.
    In all examples, \ROM{} manages to query the latent model at least $10^{5}$
    times, which is the maximal amount allowed.
    \Proj{} and \ROMProj{} manage to query the system more often than \FOM{},
    leading to speedups ranging from around $1.7$ up to $2.2$.}
  \label{fig:queries_full}
\end{figure*}
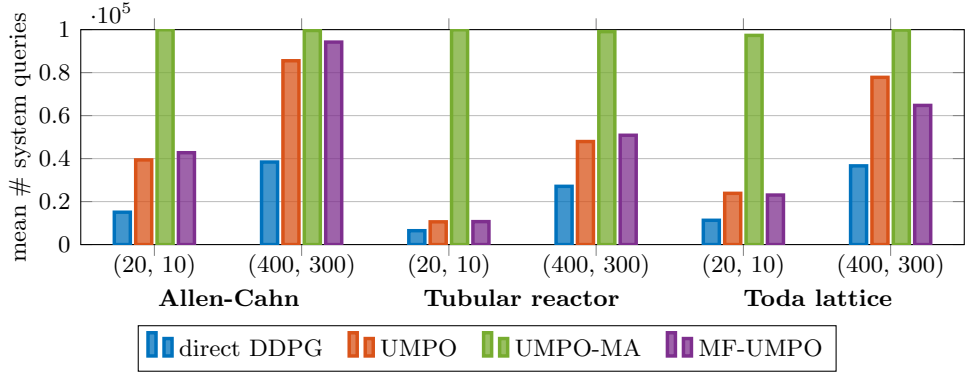

In \Cref{fig:queries_full}, the mean amount of system queries up to the
training times $0.02$\,h for small neural networks and $0.2$\,h for
large neural networks are shown.
In all examples, \ROM{} is able to (nearly) perform the maximum allowed
$10^{5}$ system queries.
Also, \Proj{} and \ROMProj{} always outperform \FOM{}, with a speed up factor
of around $1.7$ to $2.2$, which allows these methods to reach higher rewards
quicker as they can perform more training steps in a smaller amount of time.


\subsubsection{Quality of learned policies}%
\label{sec:quality}

\begin{figure*}[t]
  \centering
  \begin{subfigure}[b]{.495\textwidth}
    \centering
  \tikzexternalenable%
  \tikzsetnextfilename{tubularreactor_temp_FOM_full}%
  \begin{tikzpicture}[font = \plotfontsize, text = black]
  \begin{axis}[%
    scale only axis,
    width              = .775\linewidth,
    height             = .325\linewidth,
    xmin               = 0,
    xmax               = 30,
    ymin               = 0,
    ymax               = 1,
    xminorticks        = false,
    yminorticks        = false,
    xlabel             = {time $\tau \cdot t$},
    ylabel             = {spatial coordinates},
    ylabel style       = {yshift = -.4em},
    xlabel style       = {yshift = .4em},
    scaled x ticks     = false,
    x tick label style = {/pgf/number format/1000 sep={\,}},
    y tick label style = {/pgf/number format/1000 sep={\,}}
  ]
  
    \addplot graphics[xmin = 0, xmax = 30, ymin = 0, ymax = 1]
        {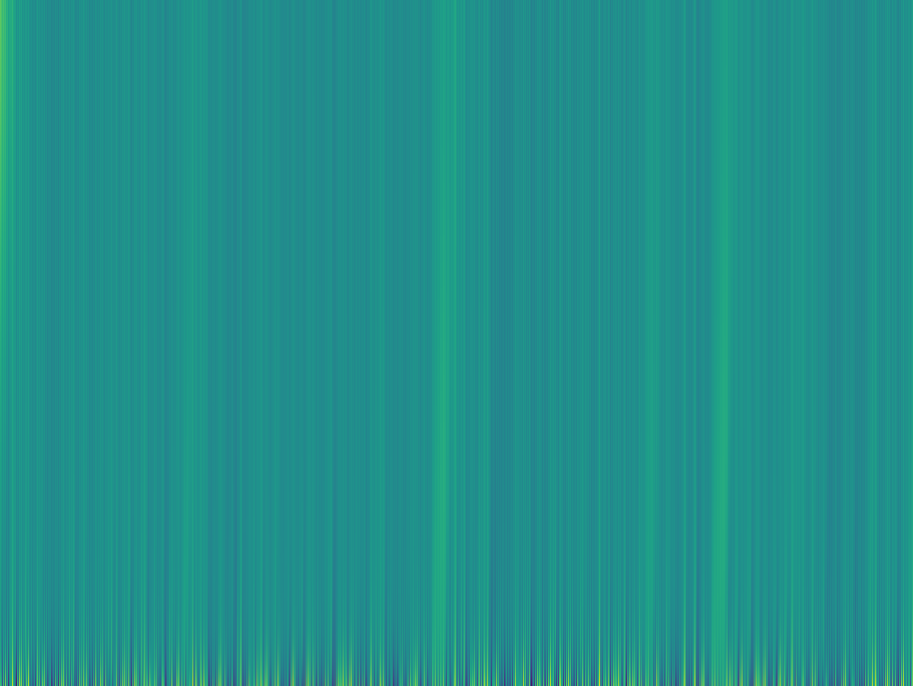};
  \end{axis}
\end{tikzpicture}%
  \tikzexternaldisable%

    \vspace{-.5\baselineskip}
    \caption{\FOM{}}
  \end{subfigure}%
  \hfill%
  \begin{subfigure}[b]{.495\textwidth}
    \centering
  \tikzexternalenable%
  \tikzsetnextfilename{tubularreactor_temp_Proj_full}%
  \begin{tikzpicture}[font = \plotfontsize, text = black]
  \begin{axis}[%
    scale only axis,
    width              = .775\linewidth,
    height             = .325\linewidth,
    xmin               = 0,
    xmax               = 30,
    ymin               = 0,
    ymax               = 1,
    xminorticks        = false,
    yminorticks        = false,
    xlabel             = {time $\tau \cdot t$},
    ylabel             = {spatial coordinates},
    ylabel style       = {yshift = -.4em},
    xlabel style       = {yshift = .4em},
    scaled x ticks     = false,
    x tick label style = {/pgf/number format/1000 sep={\,}},
    y tick label style = {/pgf/number format/1000 sep={\,}}
  ]
  
    \addplot graphics[xmin = 0, xmax = 30, ymin = 0, ymax = 1]
        {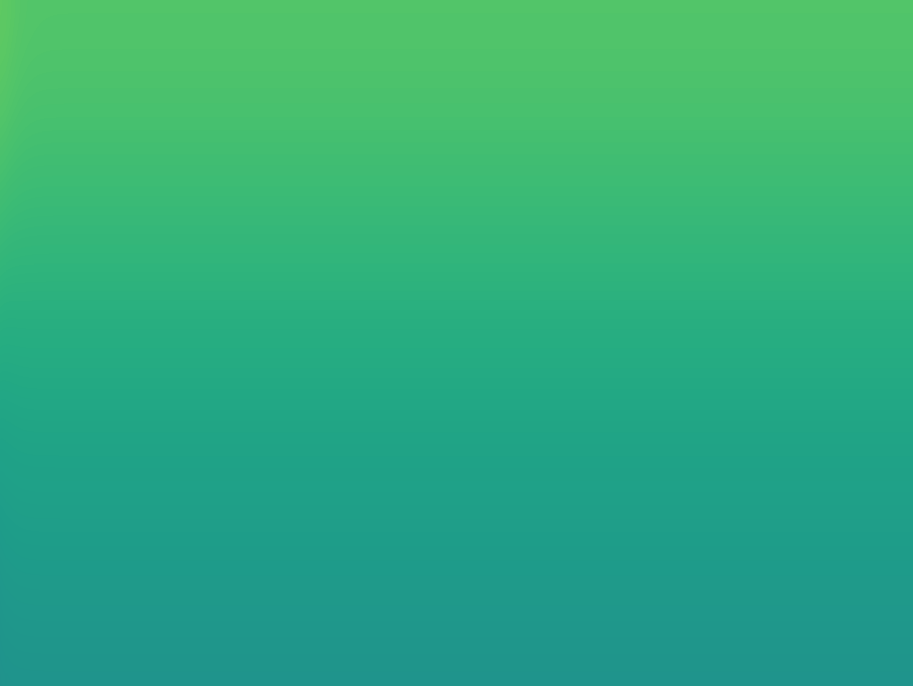};
  \end{axis}
\end{tikzpicture}%
  \tikzexternaldisable%

    \vspace{-.5\baselineskip}
    \caption{\Proj{}}
  \end{subfigure}

  \vspace{.5\baselineskip}
  \begin{subfigure}[b]{.495\textwidth}
    \centering
  \tikzexternalenable%
  \tikzsetnextfilename{tubularreactor_temp_ROM_full}%
  \begin{tikzpicture}[font = \plotfontsize, text = black] 
  \begin{axis}[%
    scale only axis,
    width              = .775\linewidth,
    height             = .325\linewidth,
    xmin               = 0,
    xmax               = 30,
    ymin               = 0,
    ymax               = 1,
    xminorticks        = false,
    yminorticks        = false,
    xlabel             = {time $\tau \cdot t$},
    ylabel             = {spatial coordinates},
    ylabel style       = {yshift = -.4em},
    xlabel style       = {yshift = .4em},
    scaled x ticks     = false,
    x tick label style = {/pgf/number format/1000 sep={\,}},
    y tick label style = {/pgf/number format/1000 sep={\,}}
  ]
  
    \addplot graphics[xmin = 0, xmax = 30, ymin = 0, ymax = 1]
        {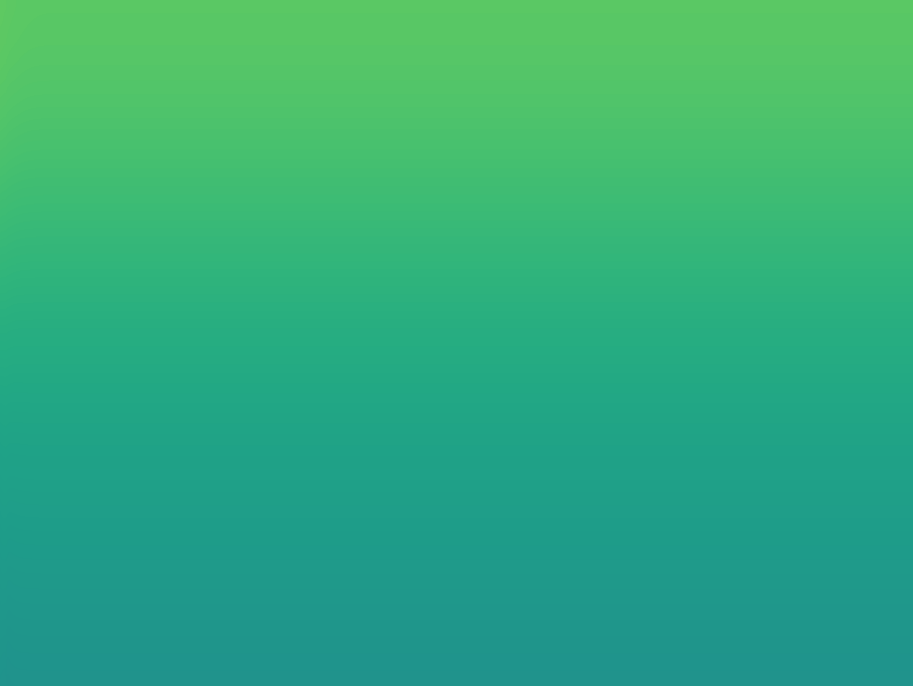};
  \end{axis}
\end{tikzpicture}%
  \tikzexternaldisable%

    \vspace{-.5\baselineskip}
    \caption{\ROM{}}
  \end{subfigure}%
  \hfill%
  \begin{subfigure}[b]{.495\textwidth}
    \centering
  \tikzexternalenable%
  \tikzsetnextfilename{tubularreactor_temp_ROMProj_full}%
  \begin{tikzpicture}[font = \plotfontsize, text = black]
  \begin{axis}[%
    scale only axis,
    width              = .775\linewidth,
    height             = .325\linewidth,
    xmin               = 0,
    xmax               = 30,
    ymin               = 0,
    ymax               = 1,
    xminorticks        = false,
    yminorticks        = false,
    xlabel             = {time $\tau \cdot t$},
    ylabel             = {spatial coordinates},
    ylabel style       = {yshift = -.4em},
    xlabel style       = {yshift = .4em},
    scaled x ticks     = false,
    x tick label style = {/pgf/number format/1000 sep={\,}},
    y tick label style = {/pgf/number format/1000 sep={\,}}
  ]
  
    \addplot graphics[xmin = 0, xmax = 30, ymin = 0, ymax = 1]
        {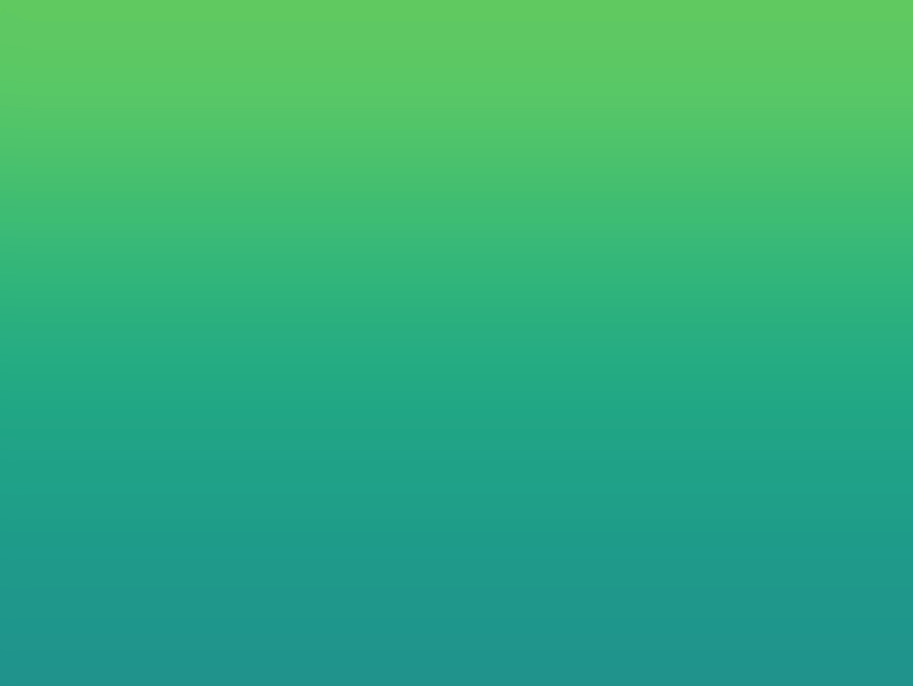};
  \end{axis}
\end{tikzpicture}%
  \tikzexternaldisable%

    \vspace{-.5\baselineskip}
    \caption{\ROMProj{}}
  \end{subfigure}

  \vspace{.5\baselineskip}
  \tikzexternalenable%
  \tikzsetnextfilename{tubularreactor_temp_legend}%
  \begin{tikzpicture}[font = \legendfontsize, text = black]
  \node[draw = none, minimum width = 0cm, inner sep = 0cm](start){};
  \node(leg) at (start.north east) [anchor = north west]{\tikz
  \begin{axis}[%
    hide axis,
    scale only axis,
    width  = 10cm,
    height = .1cm,
    point meta min = 0.7199,
    point meta max = 1.3036,
    colorbar,
    colorbar horizontal,
    colorbar style = {
      at = {(.5, 0)},
      anchor = north},
    scaled x ticks = false,
    x tick label style = {/pgf/number format/fixed}]
  \end{axis};};
  \node[draw = none, minimum width = 0cm, inner sep = 0cm](end)
    at (leg.north east) [anchor = north west]{};
\end{tikzpicture}%
  \tikzexternaldisable%

  \caption{Temperature profiles of the tubular reactor with policies obtained
    after $0.2$\,h of training time with neural network architecture
    $(400, 300)$.
    The plots show that \FOM{} fails to stabilize the reactor where as \Proj{}
    provides a stabilizing policy, with \ROMProj{} achieving lowest
    oscillations and deriving the system closest to the desired steady state.}
  \label{fig:tubularreactor_temp_full}
\end{figure*}
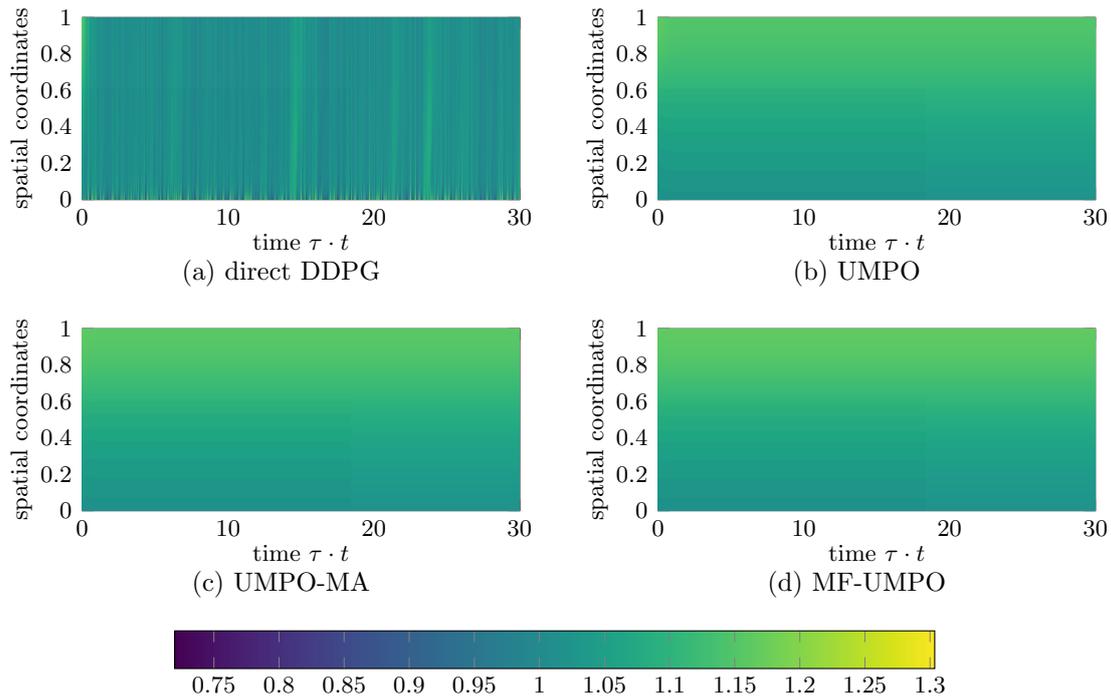

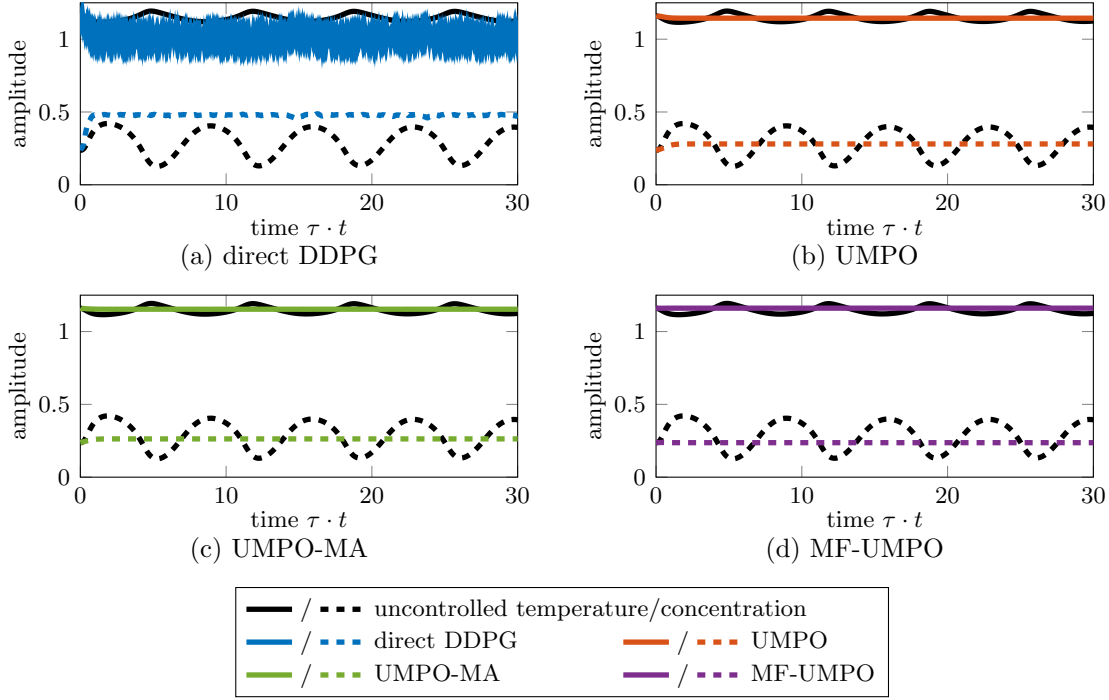
\begin{figure*}[t]
  \centering
  \begin{subfigure}[b]{.495\textwidth}
    \centering
  \tikzexternalenable%
  \tikzsetnextfilename{tubularreactorFOM}%
  \begin{tikzpicture}[font = \plotfontsize, text = black]
  \pgfplotstableread{graphics/data/tubularreactorFOM.dat}\tableSIM

  \begin{axis}[
    scale only axis,
    width        = .775\linewidth,
    height       = .325\linewidth,
    xmin         = 0,
    xmax         = 30,
    ymin         = 0,
    ymax         = 1.25,
    xminorticks  = false,
    yminorticks  = false,
    xlabel       = {time $\tau \cdot t$},
    ylabel       = {amplitude},
    ylabel style = {yshift = -.4em},
    xlabel style = {yshift = .4em}
  ]
  
    \addplot+[lines, black, mark = none]
      table[x index = 0, y index = 1] {\tableSIM};
    \addplot+[lines, black, mark = none, dashed]
      table[x index = 0, y index = 2] {\tableSIM};

    \addplot+[lines, matlabblue, mark = none]
      table[x index = 0, y index = 3] {\tableSIM};
    \addplot+[lines, matlabblue, mark = none, dashed]
      table[x index = 0, y index = 4] {\tableSIM};
  \end{axis}
\end{tikzpicture}%
  \tikzexternaldisable%

    \vspace{-.5\baselineskip}
    \caption{\FOM{}}
    \label{fig:tubularreactor_sim_FOM}
  \end{subfigure}%
  \hfill%
  \begin{subfigure}[b]{.495\textwidth}
    \centering
  \tikzexternalenable%
  \tikzsetnextfilename{tubularreactorProj}%
  \begin{tikzpicture}[font = \plotfontsize, text = black]
  \pgfplotstableread{graphics/data/tubularreactorProj.dat}\tableSIM

  \begin{axis}[
    scale only axis,
    width        = .775\linewidth,
    height       = .325\linewidth,
    xmin         = 0,
    xmax         = 30,
    ymin         = 0,
    ymax         = 1.25,
    xminorticks  = false,
    yminorticks  = false,
    xlabel       = {time $\tau \cdot t$},
    ylabel       = {amplitude},
    ylabel style = {yshift = -.4em},
    xlabel style = {yshift = .4em}
  ]
  
    \addplot+[lines, black, mark = none]
      table[x index = 0, y index = 1] {\tableSIM};
    \addplot+[lines, black, mark = none, dashed]
      table[x index = 0, y index = 2] {\tableSIM};

    \addplot+[lines, matlaborange, mark = none]
      table[x index = 0, y index = 3] {\tableSIM};
    \addplot+[lines, matlaborange, mark = none, dashed]
      table[x index = 0, y index = 4] {\tableSIM};
  \end{axis}
\end{tikzpicture}%
  \tikzexternaldisable%

    \vspace{-.5\baselineskip}
    \caption{\Proj{}}
    \label{fig:tubularreactor_sim_Proj}
  \end{subfigure}

  \vspace{.5\baselineskip}
  \begin{subfigure}[b]{.495\textwidth}
    \centering
  \tikzexternalenable%
  \tikzsetnextfilename{tubularreactorROM}%
  \begin{tikzpicture}[font = \plotfontsize, text = black]
  \pgfplotstableread{graphics/data/tubularreactorROM.dat}\tableSIM

  \begin{axis}[
    scale only axis,
    width        = .775\linewidth,
    height       = .325\linewidth,
    xmin         = 0,
    xmax         = 30,
    ymin         = 0,
    ymax         = 1.25,
    xminorticks  = false,
    yminorticks  = false,
    xlabel       = {time $\tau \cdot t$},
    ylabel       = {amplitude},
    ylabel style = {yshift = -.4em},
    xlabel style = {yshift = .4em}
  ]
  
    \addplot+[lines, black, mark = none]
      table[x index = 0, y index = 1] {\tableSIM};
    \addplot+[lines, black, mark = none, dashed]
      table[x index = 0, y index = 2] {\tableSIM};

    \addplot+[lines, matlabgreen, mark = none]
      table[x index = 0, y index = 3] {\tableSIM};
    \addplot+[lines, matlabgreen, mark = none, dashed]
      table[x index = 0, y index = 4] {\tableSIM};
  \end{axis}
\end{tikzpicture}%
  \tikzexternaldisable%

    \vspace{-.5\baselineskip}
    \caption{\ROM{}}
    \label{fig:tubularreactor_sim_ROM}
  \end{subfigure}%
  \hfill%
  \begin{subfigure}[b]{.495\textwidth}
    \centering
  \tikzexternalenable%
  \tikzsetnextfilename{tubularreactorROMProj}%
  \begin{tikzpicture}[font = \plotfontsize, text = black]
  \pgfplotstableread{graphics/data/tubularreactorROMProj.dat}\tableSIM

  \begin{axis}[
    scale only axis,
    width        = .775\linewidth,
    height       = .325\linewidth,
    xmin         = 0,
    xmax         = 30,
    ymin         = 0,
    ymax         = 1.25,
    xminorticks  = false,
    yminorticks  = false,
    xlabel       = {time $\tau \cdot t$},
    ylabel       = {amplitude},
    ylabel style = {yshift = -.4em},
    xlabel style = {yshift = .4em}
  ]
  
    \addplot+[lines, black, mark = none]
      table[x index = 0, y index = 1] {\tableSIM};
    \addplot+[lines, black, mark = none, dashed]
      table[x index = 0, y index = 2] {\tableSIM};

    \addplot+[lines, matlabpurple, mark = none]
      table[x index = 0, y index = 3] {\tableSIM};
    \addplot+[lines, matlabpurple, mark = none, dashed]
      table[x index = 0, y index = 4] {\tableSIM};
  \end{axis}
\end{tikzpicture}%
  \tikzexternaldisable%

    \vspace{-.5\baselineskip}
    \caption{\ROMProj{}}
    \label{fig:tubularreactor_sim_ROMProj}
  \end{subfigure}

  \vspace{.5\baselineskip}
  \tikzexternalenable%
  \tikzsetnextfilename{tubularreactor_legend}%
  \begin{tikzpicture}[font = \legendfontsize, text = black]
  \begin{axis}[%
    hide axis,
    width  = 1mm,
    height = 1mm,
    scale only axis,
    xmin = 0,
    xmax = 1,
    ymin = 0,
    ymax = 1,
    legend columns = 4, 
    legend style   = {
      at     = {(0,0)},
      anchor = center,
      /tikz/every even column/.append style = {column sep = 0cm}},
    legend cell align  = {left},
    clip mode          = individual
  ]

    \addlegendimage{lines, black, mark = none}
    \addlegendentry{/}
    \addlegendimage{lines, black, mark = none, dashed}
    \addlegendentry{\makebox[0pt][l]{uncontrolled temperature/concentration}}

    \addlegendimage{empty legend}
    \addlegendentry{}
    \addlegendimage{empty legend}
    \addlegendentry{}

    \addlegendimage{lines, matlabblue, mark = none}
    \addlegendentry{/}
    \addlegendimage{lines, matlabblue, mark = none, dashed}
    \addlegendentry{\FOM{}\hspace{2.7em}\phantom{pP}}

    \addlegendimage{lines, matlaborange, mark = none}
    \addlegendentry{/}
    \addlegendimage{lines, matlaborange, mark = none, dashed}
    \addlegendentry{\Proj{}}

    \addlegendimage{lines, matlabgreen, mark = none}
    \addlegendentry{/}
    \addlegendimage{lines, matlabgreen, mark = none, dashed}
    \addlegendentry{\ROM{}}

    \addlegendimage{lines, matlabpurple, mark = none}
    \addlegendentry{/}
    \addlegendimage{lines, matlabpurple, mark = none, dashed}
    \addlegendentry{\ROMProj{}}
  \end{axis}
\end{tikzpicture}%
  \tikzexternaldisable%

  \caption{The plots show the outputs of the simulations of the tubular reactor
    with a policy that is obtained after $0.2$\,h of training time
    with neural network architecture $(400, 300)$.
    The \FOM{} approach fails to stabilize the system within the short training
    time of $0.2$\,h because neither temperature nor concentration measurements
    are close to the steady state.
    Additionally, the temperature strongly oscillates.
    The other approaches learn a stabilizing policy within the training
    time of $0.2$\,h.
    Notice that deviations from the steady state decrease in the order of
    \Proj{}, \ROM{} and \ROMProj{}; thus the multi-fidelity
    approach \ROMProj{} achieves a stabilization closest
    to the desired steady state.}
  \label{fig:tubularreactor_sim}
\end{figure*}

To illustrate the results of applying the learned policies, we consider the
tubular reactor example and the best performing policies for the
neural network size $(400, 300)$ for the policy after $0.2$\,h of
training time.
In the following experiments, the initial condition is set to the steady state,
$x(0) = \xs$ and a random disturbance is applied up to the physical time
$\tau \cdot t = 0.3$ to trigger the system instabilities.
We plot the temperature profiles over time in
\Cref{fig:tubularreactor_temp_full}.
A stabilized system leads to a smooth temperature transition over the
space-time domain.
The states are strongly oscillating for \FOM{}, while the proposed methods
based on the unstable manifolds provide smooth, stabilized simulations.
To clearer see the oscillations, we plot the temperature and concentration at
the left end of the tube over time in \Cref{fig:tubularreactor_sim}.
The policy learned by \FOM{} steers the outputs away from the uncontrolled
oscillations but is clearly not stabilizing towards the considered steady
state, as the concentration deviates far from the initial condition and the
temperature is strongly oscillating.
In contrast, the policies obtained by the newly proposed methods are
stabilizing the system.
We can see minor deviations from the steady state in
\Cref{fig:tubularreactor_sim_Proj}, which become smaller moving to
\Cref{fig:tubularreactor_sim_ROM} and finally disappear in
\Cref{fig:tubularreactor_sim_ROMProj}.
The most stable results are given by \ROMProj{}, with an at least three orders
of magnitude smaller deviation from the desired state when compared to \FOM{}.


\subsubsection{Quantitative comparison}%
\label{sec:quantitative}

\begin{figure*}[t]
  \centering
  \tikzexternalenable%
  \tikzsetnextfilename{time_per_step_rewards_full}%
  \begin{tikzpicture}
  \node (time) {%
    \begin{tikzpicture}[font = \plotfontsize, text = black]
      \begin{axis}[
        scale only axis,
        width             = .775\linewidth,
        height            = .19\linewidth,
        ybar,
        bar width         = 6pt,
        ymin              = 0,
        ymax              = 0.0275,
        ylabel            = {training time per step (s)},
        ymajorgrids,
        ylabel style      = {yshift = -.4em},
        xtick             = data,
        symbolic x coords = {test11, test12, test21, test22, test31, test32},
        xticklabel style  = {yshift = .4em},
        point meta=explicit symbolic,
        clip = false,
        xticklabels       = {(20{,} 10), (400{,} 300), (20{,} 10), (400{,} 300),
                              (20{,} 10), (400{,} 300)},
        legend columns    = -1,
        legend cell align = {left},
        legend style      = {
          at                                    = {(.46,-0.4)},
          anchor                                = north,
          /tikz/every even column/.append style = {column sep = 0.25cm}},
        cycle list name = barlist
      ]

        \addplot+[ybar] coordinates{
          (test11, 0.004771)
          (test12, 0.017625)
          (test21, 0.010516)
          (test22, 0.025099)
          (test31, 0.006884)
          (test32, 0.020311)
        };

        \addplot+[ybar] coordinates{
          (test11, 0.001797)
          (test12, 0.008183)
          (test21, 0.006477)
          (test22, 0.013398)
          (test31, 0.003156)
          (test32, 0.010510)
        };

        \addplot+[ybar] coordinates{
          (test11, 0.000655)
          (test12, 0.006305)
          (test21, 0.000685)
          (test22, 0.006757)
          (test31, 0.000700)
          (test32, 0.007403)
        };

        \addplot+[ybar] coordinates{
          (test11, 0.002388)
          (test12, 0.014245)
          (test21, 0.007135)
          (test22, 0.020947)
          (test31, 0.004210)
          (test32, 0.017586)
        };
      \end{axis}

      \node at (1.95, -.7) [anchor = center] {\textbf{Allen-Cahn}};
      \node at (5.8, -.7) [anchor = center] {\textbf{Tubular reactor}};
      \node at (9.7, -.7) [anchor = center] {\textbf{Toda lattice}};
    \end{tikzpicture}};

  \node(rewards) [below = 0\baselineskip of time.south east,
    anchor = north east] {%
    \begin{tikzpicture}[font = \plotfontsize, text = black]
      \begin{axis}[
        scale only axis,
        width             = .775\linewidth,
        height            = .19\linewidth,
        ybar,
        bar width         = 6pt,
        ymode             = log,
        ymin              = 1e-4,
        ymax              = 1e+0,
        y dir             = reverse,
        ylabel            = {$\max(\logmean(\Rn))$},
        ymajorgrids,
        ylabel style      = {yshift = -.4em},
        xtick             = data,
        symbolic x coords = {test11, test12, test21, test22, test31, test32},
        yticklabels       = {x, $-10^{-3}$, $-10^{-1}$},
        xticklabel style  = {yshift = .4em},
        point meta        = explicit symbolic,
        clip              = false,
        xticklabels       = {(20{,} 10), (400{,} 300), (20{,} 10), (400{,} 300),
                              (20{,} 10), (400{,} 300)},
        legend columns    = -1,
        legend cell align = {left},
        legend style      = {
          at                                    = {(.46, -0.375)},
          anchor                                = north,
          /tikz/every even column/.append style = {column sep = 0.25cm}},
        cycle list name = barlist
      ]

        \addplot+[ybar] coordinates{
          (test11, 3.548515e-02)
          (test12, 1.131312e-01)
          (test21, 3.140066e-01)
          (test22, 8.526817e-01)
          (test31, 5.909171e-01)
          (test32, 4.462692e-01)
        };
        \addlegendentry{\legendfontsize \FOM{}}

        \addplot+[ybar] coordinates{
          (test11, 3.851234e-02)
          (test12, 7.241075e-03)
          (test21, 7.055078e-01)
          (test22, 3.404605e-01)
          (test31, 3.247149e-01)
          (test32, 4.881409e-01)
        };
        \addlegendentry{\legendfontsize \Proj{}}

        \addplot+[ybar] coordinates{
          (test11, 4.176828e-03)
          (test12, 4.327097e-04)
          (test21, 8.688416e-03)
          (test22, 3.680587e-02)
          (test31, 3.883624e-02)
          (test32, 1.652134e-02)
        };
        \addlegendentry{\legendfontsize \ROM{}}

        \addplot+[ybar] coordinates{
          (test11, 6.945772e-04)
          (test12, 2.170322e-04)
          (test21, 4.344140e-03)
          (test22, 2.468115e-03)
          (test31, 6.658587e-01)
          (test32, 2.764250e-01)
        };
        \addlegendentry{\legendfontsize \ROMProj{}}
      \end{axis}

      \node at (1.95, -.7) [anchor = center] {\textbf{Allen-Cahn}};
      \node at (5.8, -.7) [anchor = center] {\textbf{Tubular reactor}};
      \node at (9.7, -.7) [anchor = center] {\textbf{Toda lattice}};
    \end{tikzpicture}};
\end{tikzpicture}%
  \tikzexternaldisable%

  \caption{Top shows comparison of the mean training times per step.
    For all examples, \ROM{} has the lowest training time per step that is
    similar across all examples for the same neural network sizes.
    The method \Proj{} has the next larger training time per step, followed by
    \ROMProj{}.
    All new approaches require shorter training times than the
    classical \FOM{} method. Bottom shows  comparison of the best logarithmic
    means of normalized accumulated rewards after $0.02$\,h and $0.2$\,h of
    training time for small and large neural networks:
    In all examples, the \ROM{} initialization achieves up to two orders of
    magnitude better rewards than \FOM{}.
    \ROMProj{} further improves these rewards with the exception of the
    Toda lattice model, which contains a challenging nonlinear term, which
    affects the performance of the \ROMProj{} policy and due to which the
    dynamics blow up after only few time steps in the presence of
    instabilities.}
  \label{fig:time_per_step}
\end{figure*}
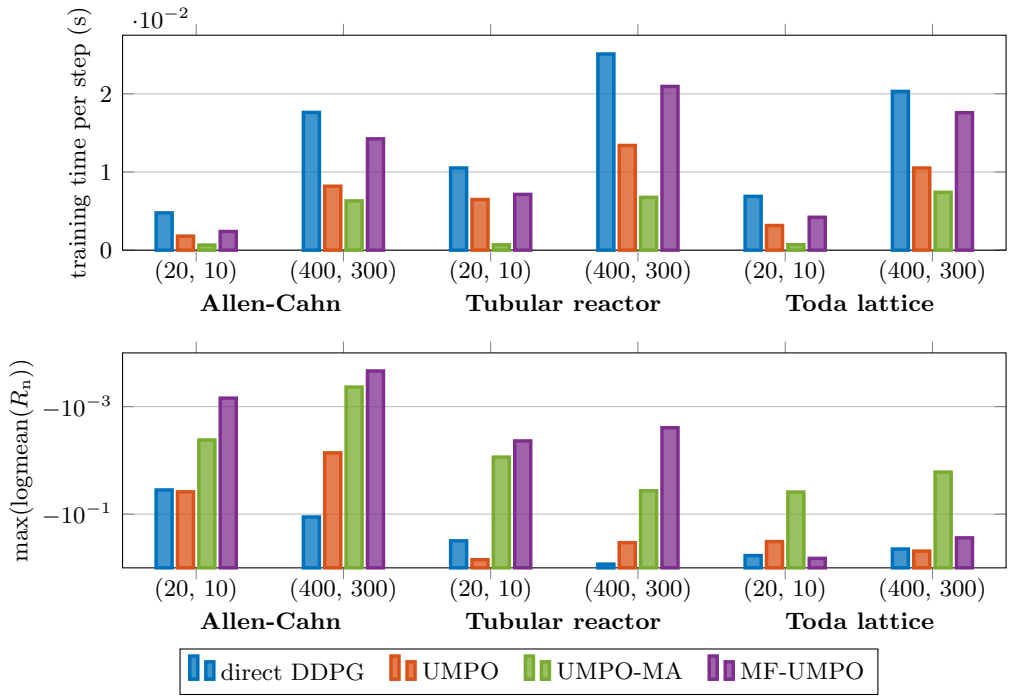

The plots in \Cref{fig:time_per_step} show the training time per
learning step and the maximum reward at the end of the learning for the
Allen-Cahn, the tubular reactor, and Toda lattice example, respectively.
The shown results are averages over the five policies of each network size
that stabilize the system closest to the steady state.
As can be seen in the top plot of \Cref{fig:time_per_step}, all new
approaches show an improvement in terms of their training time per step
compared to the classical \FOM{} method.
The time shown for \ROMProj{} contains the training of the initial
phase with \ROM{}, but it is still faster for both neural network sizes
compared to \FOM{}.
The bottom plot shows the best mean rewards that have
been reached after $0.02$\,h of training time using $(20, 10)$ networks and
after $0.2$\,h for the $(400, 300)$ networks.
For the smaller networks in the Allen-Cahn and tubular reactor examples,
\Proj{} does not surpass \FOM{} but yields similar rewards.
\Proj{} might improve for longer training times; however, since
\Proj{} provides faster training steps, the smaller neural networks may
still be preferred here despite the somewhat lower rewards.
\ROM{} and \ROMProj{} are surpassing \Proj{} and \FOM{} in the Allen-Cahn and
tubular reactor examples by at least one and up to two orders of magnitude.
The behavior of the methods changes in the Toda lattice example.
\ROM{} provides here the best rewards, which \ROMProj{} is not able to preserve
or to improve on despite using \ROM{} as initialization.
A possible explanation for these results comes from the strength of the
nonlinearity in the example.
Even small disturbances can have a catastrophic effect on the states of the
system leading to blowups after only few time steps.
Such disturbances are difficult to avoid in the training phase of reinforcement
learning via stochastic optimization methods, which means that large portions
of the data obtained from the nonlinear system for training the learning agents
are in unstable regimes already and thus unsuited for training.
However, we can see that for the $(20, 10)$ networks, \Proj{} yields
higher rewards than \FOM{} and, for $(400, 300)$, \ROMProj{} gives
higher rewards than \FOM{}.


\subsubsection{Dependence of performance on the state dimension}%
\label{sec:dimension}

\begin{figure*}[t]
  \centering
  \begin{subfigure}[b]{.495\linewidth}
    \centering
  \tikzexternalenable%
  \tikzsetnextfilename{training_times_1}%
  \begin{tikzpicture}[
  fill between/on layer = {main},
  font                  = \plotfontsize,
  text                  = black
]
  \pgfplotstableread{graphics/data/training_times_net=20_10.dat}\tableTIME
  
  \begin{axis}[%
    scale only axis,
    width           = .775\linewidth,
    height          = .325\linewidth,
    xmin            = 800,
    xmax            = 6200,
    ymin            = -0.2,
    ymax            = 2.8,
    xminorticks     = false,
    yminorticks     = false,
    xtick           = {998, 1998, 3998, 5998},
    xlabel          = {state dimension $N$\vphantom{Pp}},
    ylabel          = {training time (h)},
    ylabel style    = {yshift = -.4em},
    xlabel style    = {yshift = .4em},
    cycle list name = timelist
  ]

    \addplot+[name path = A] table[x index = 0, y index = 1] {\tableTIME};
    \addplot+[name path = B] table[x index = 0, y index = 2] {\tableTIME};
    \addplot+ fill between [of = A and B];
    \addplot+[mark = *, mark options = {fill = matlabblue}]
      table[x index = 0, y index = 3] {\tableTIME};

    \addplot+[name path = A] table[x index = 0, y index = 4] {\tableTIME};
    \addplot+[name path = B] table[x index = 0, y index = 5] {\tableTIME};
    \addplot+ fill between [of = A and B];
    \addplot+[mark = square*, mark options = {fill = matlaborange}]
      table[x index = 0, y index = 6] {\tableTIME};

    \addplot+[name path = A] table[x index = 0, y index = 7] {\tableTIME};
    \addplot+[name path = B] table[x index = 0, y index = 8] {\tableTIME};
    \addplot+ fill between [of = A and B];
    \addplot+[mark = triangle*, mark options = {fill = matlabgreen}]
      table[x index = 0, y index = 9] {\tableTIME};

    \addplot+[name path = A] table[x index = 0, y index = 10] {\tableTIME};
    \addplot+[name path = B] table[x index = 0, y index = 11] {\tableTIME};
    \addplot+ fill between [of = A and B];
    \addplot+[mark = diamond*, mark options = {fill = matlabpurple}]
      table[x index = 0, y index = 12] {\tableTIME};
  \end{axis}
\end{tikzpicture}%
  \tikzexternaldisable%

    \vspace{-.5\baselineskip}
    \caption{network architecture $(20, 10)$}
  \end{subfigure}%
  \hfill%
  \begin{subfigure}[b]{.495\linewidth}
    \centering
  \tikzexternalenable%
  \tikzsetnextfilename{training_times_2}%
  \begin{tikzpicture}[
  fill between/on layer = {main},
  font                  = \plotfontsize,
  text                  = black
]
  \pgfplotstableread{graphics/data/training_times_net=400_300.dat}\tableTIME
  
  \begin{axis}[%
    scale only axis,
    width           = .775\linewidth,
    height          = .325\linewidth,
    xmin            = 800,
    xmax            = 6200,
    ymin            = -0.2,
    ymax            = 6,
    xminorticks     = false,
    yminorticks     = false,
    xtick           = {998, 1998, 3998, 5998},
    xlabel          = {state dimension $N$\vphantom{Pp}},
    ylabel          = {training time (h)},
    ylabel style    = {yshift = -.4em},
    xlabel style    = {yshift = .4em},
    cycle list name = timelist
  ]

    \addplot+[name path = A] table[x index = 0, y index = 1] {\tableTIME};
    \addplot+[name path = B] table[x index = 0, y index = 2] {\tableTIME};
    \addplot+ fill between [of = A and B];
    \addplot+[mark = *, mark options = {fill = matlabblue}]
      table[x index = 0, y index = 3] {\tableTIME};

    \addplot+[name path = A] table[x index = 0, y index = 4] {\tableTIME};
    \addplot+[name path = B] table[x index = 0, y index = 5] {\tableTIME};
    \addplot+ fill between [of = A and B];
    \addplot+[mark = square*, mark options = {fill = matlaborange}]
      table[x index = 0, y index = 6] {\tableTIME};

    \addplot+[name path = A] table[x index = 0, y index = 7] {\tableTIME};
    \addplot+[name path = B] table[x index = 0, y index = 8] {\tableTIME};
    \addplot+ fill between [of = A and B];
    \addplot+[mark = triangle*, mark options = {fill = matlabgreen}]
      table[x index = 0, y index = 9] {\tableTIME};

    \addplot+[name path = A] table[x index = 0, y index = 10] {\tableTIME};
    \addplot+[name path = B] table[x index = 0, y index = 11] {\tableTIME};
    \addplot+ fill between [of = A and B];
    \addplot+[mark = diamond*, mark options = {fill = matlabpurple}]
      table[x index = 0, y index = 12] {\tableTIME};
  \end{axis}
\end{tikzpicture}%
  \tikzexternaldisable%

    \vspace{-.5\baselineskip}
    \caption{network architecture $(400, 300)$}
  \end{subfigure}

  \vspace{.5\baselineskip}
  \tikzexternalenable%
  \tikzsetnextfilename{training_times_legend}%
  \begin{tikzpicture}[font = \legendfontsize, text = black]
  \begin{axis}[%
    hide axis,
    width  = 1mm,
    height = 1mm,
    scale only axis,
    xmin = 0,
    xmax = 1,
    ymin = 0,
    ymax = 1,
    legend columns = -1, 
    legend style   = {
      at     = {(0,0)},
      anchor = center,
      /tikz/every even column/.append style = {column sep = 0.4cm}},
    legend cell align  = {left},
    clip mode          = individual,
    cycle list name    = timelist
  ]

    \pgfplotsset{cycle list shift = 3}
    \addplot+[mark = *, mark options = {fill = matlabblue}]
      coordinates{(0, 0)};
    \addlegendentry{\FOM{}}

    \pgfplotsset{cycle list shift = 6}
    \addplot+[mark = square*, mark options = {fill = matlaborange}]
      coordinates{(0, 0)};
    \addlegendentry{\Proj{}}
  
    \pgfplotsset{cycle list shift = 9}
    \addplot+[mark = triangle*, mark options = {fill = matlabgreen}]
      coordinates{(0, 0)};
    \addlegendentry{\ROM{}}

    \pgfplotsset{cycle list shift = 12}
    \addplot+[mark = diamond*, mark options = {fill = matlabpurple}]
      coordinates{(0, 0)};
    \addlegendentry{\ROMProj{}}
  \end{axis}
\end{tikzpicture}%
  \tikzexternaldisable%

  \caption{The training times of the \FOM{}, \Proj{}, \ROM{}, \ROMProj{}
    approaches are compared for different state-space dimensions of the
    tubular reactor example.
    The speedup obtained is about $20-25\%$ for the small networks and
    $70-75\%$ for the large networks.
    Notice that the training time of \ROM{} is independent of $\nh$ and
    therefore is constant in the shown plots.}
  \label{fig:training_times_full}
\end{figure*}
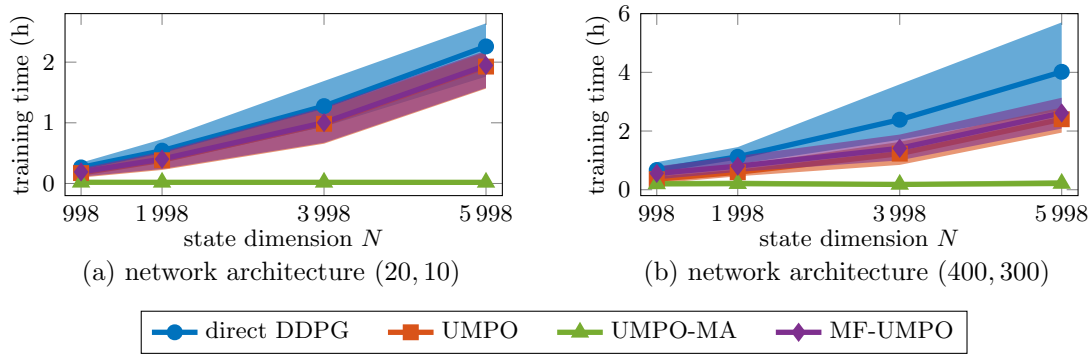

\Cref{fig:training_times_full} shows a comparison of the training times
for the two neural network sizes $(20, 10)$ and $(400, 300)$ when applied to
finer and finer spatial discretizations of the tubular reactor model, which
directly corresponds to the state dimension $\nh$ of the system that is to be
stabilized.
For both neural network sizes, we can observe the performance improvement
of the new methods compared to \FOM{}.
For the smaller networks, the training time is clearly dominated by the
evaluation time of the high-dimensional system~\cref{eqn:sys} rather than
the optimization of the neural networks such that \Proj{} and \ROMProj{} give
only an improvement of about $20-25\%$ compared to \FOM{}.
This factor increases to around $70-75\%$ for larger neural networks, which
indicates that the proposed methods have increasing benefit as the network
sizes grow.
The training times for \ROM{} do not depend on the full state dimension $\nh$
but on the number of unstable system modes $\nr = 2$, which is constant even as
discretizations are refined.
Therefore, these training times are constant up to measurement fluctuations in
\Cref{fig:training_times_full}.
The training times of \ROM{} mildly increase for the larger neural networks due
to the increase in parameters to learn, which is still negligible small
compared to \FOM{}, where the number of parameters in the neural network
parametrization scales with the state dimension $\nh$.
This demonstrates the power of the unstable manifold with a dimension that is
fixed even if the solution fields of the underlying governing equations are
ever finer resolved and thus the state dimension $\nh$ of the dynamical
system~\cref{eqn:sys} grows.


\section{Conclusions}%
\label{sec:conclusions}

In this work, we showed that the latent manifold of unstable dynamics is well
suited for finding stabilizing policies with reinforcement learning.
The unstable latent manifold can be leveraged either directly with
high-dimensional systems or with a two-step, multi-fidelity approach that first
trains on latent models and then fine-tunes and ultimately certifies the learned
policies on the high-dimensional systems.
The proposed approaches based on the unstable latent manifold achieve rewards
that are several orders of magnitude higher than the rewards achieved with
classical approaches without the unstable latent manifold within a given fixed
training time.

While the new methods have been shown to work well in the presented numerical
examples, there are limitations to them.
We considered a small number of unstable system modes, which allowed
to consider low-dimensional parametrizations of the policy.
This is not the case in all practical applications.
In particular for chaotic systems, where the dimension of the unstable
manifold scales with $\Ocal(\nh)$, the new approach is not expected
to yield performance improvements.


\section*{Acknowledgments}%
\addcontentsline{toc}{section}{Acknowledgments}

The authors were supported by the Air Force Office of Scientific Research
(AFOSR) award FA9550-21-1-0222 (Dr. Fariba Fahroo) and the second author was
additionally supported by the National Science Foundation under Grant
No.~2012250 and No.~2046521.
Parts of this work were carried out while the first author was with the
Courant Institute of Mathematical Sciences, New York University, USA.


\addcontentsline{toc}{section}{References}
\bibliographystyle{plainurl}
\bibliography{bibtex/myref}

\begin{thebibliography}{10}

\bibitem{AbaBCetal16}
M.~Abadi, P.~Barham, J.~Chen, Z.~Chen, A.~Davis, J.~Dean, M.~Devin,
  S.~Ghemawat, G.~Irving, M.~Isard, M.~Kudlur, J.~Levenberg, R.~Monga,
  S.~Moore, D.~G. Murray, B.~Steiner, P.~Tucker, V.~Vasudevan, P.~Warden,
  M.~Wicke, Y.~Yu, and X.~Zheng.
\newblock {TensorFlow}: {A} system for large-scale machine learning.
\newblock In {\em Proceedings of the 12th USENIX Conference on Operating
  Systems Design and Implementation (OSDI'16)}, pages 265--283, 2016.

\bibitem{AllC75}
S.~M. Allen and J.~W. Cahn.
\newblock Coherent and incoherent equilibria in iron-rich iron-aluminum alloys.
\newblock {\em Acta Metall.}, 23(9):1017--1026, 1975.
\newblock \href {https://doi.org/10.1016/0001-6160(75)90106-6}
  {\path{doi:10.1016/0001-6160(75)90106-6}}.

\bibitem{AntGI16}
A.~C. Antoulas, I.~V. Gosea, and A.~C. Ionita.
\newblock Model reduction of bilinear systems in the {L}oewner framework.
\newblock {\em {SIAM} J. Sci. Comput.}, 38(5):B889--B916, 2016.
\newblock \href {https://doi.org/10.1137/15M1041432}
  {\path{doi:10.1137/15M1041432}}.

\bibitem{Arn51}
W.~E. Arnoldi.
\newblock The principle of minimized iterations in the solution of the matrix
  eigenvalue problem.
\newblock {\em Quart. Appl. Math.}, 9:17--29, 1951.
\newblock \href {https://doi.org/10.1090/qam/42792}
  {\path{doi:10.1090/qam/42792}}.

\bibitem{Bal12}
P.~Baldi.
\newblock Autoencoders, unsupervised learning, and deep architectures.
\newblock {\em Proceedings of ICML Workshop on Unsupervised and Transfer
  Learning, PMLR}, 27:37--49, 2012.
\newblock URL: \url{https://proceedings.mlr.press/v27/baldi12a.html}.

\bibitem{BarF22}
J.~Barnett and C.~Farhat.
\newblock Quadratic approximation manifold for mitigating the {K}olmogorov
  barrier in nonlinear projection-based model order reduction.
\newblock {\em J. Comput. Phys.}, 464:111348, 2022.
\newblock \href {https://doi.org/10.1016/j.jcp.2022.111348}
  {\path{doi:10.1016/j.jcp.2022.111348}}.

\bibitem{BenGW15}
P.~Benner, S.~Gugercin, and K.~Willcox.
\newblock A survey of projection-based model reduction methods for parametric
  dynamical systems.
\newblock {\em {SIAM} Rev.}, 57(4):483--531, 2015.
\newblock \href {https://doi.org/10.1137/130932715}
  {\path{doi:10.1137/130932715}}.

\bibitem{BenHW22}
P.~Benner, J.~Heiland, and S.~W.~R. Werner.
\newblock Robust output-feedback stabilization for incompressible flows using
  low-dimensional {$\mathcal{H}_{\infty}$}-controllers.
\newblock {\em Comput. Optim. Appl.}, 82(1):225--249, 2022.
\newblock \href {https://doi.org/10.1007/s10589-022-00359-x}
  {\path{doi:10.1007/s10589-022-00359-x}}.

\bibitem{BerP24}
J.~Berman and B.~Peherstorfer.
\newblock {CoLoRA}: {C}ontinuous low-rank adaptation for reduced implicit
  neural modeling of parameterized partial differential equations.
\newblock e-print 2402.14646, arXiv, 2024.
\newblock Machine Learning (cs.LG).
\newblock \href {https://doi.org/10.48550/arXiv.2402.14646}
  {\path{doi:10.48550/arXiv.2402.14646}}.

\bibitem{BraFHetal18}
J.~Bradbury, R.~Frostig, P.~Hawkins, M.~J. Johnson, C.~Leary, D.~Maclaurin,
  G.~Necula, A.~Paszke, J.~VanderPlas, S.~Wanderman-Milne, and Q.~Zhang.
\newblock {JAX}: composable transformations of {P}ython+{N}um{P}y programs
  (version 0.3.13), May 2022.
\newblock URL: \url{https://github.com/google/jax}.

\bibitem{BroT11}
H.~Broer and F.~Takens.
\newblock {\em Dynamical Systems and Chaos}, volume 172 of {\em Applied
  Mathematical Sciences}.
\newblock Springer, New York, NY, 2011.
\newblock \href {https://doi.org/10.1007/978-1-4419-6870-8}
  {\path{doi:10.1007/978-1-4419-6870-8}}.

\bibitem{BruBPetal16}
S.~L. Brunton, B.~W. Brunton, J.~L. Proctor, and J.~N. Kutz.
\newblock {K}oopman invariant subspaces and finite linear representations of
  nonlinear dynamical systems for control.
\newblock {\em PLoS ONE}, 11(2):e0150171, 2016.
\newblock \href {https://doi.org/10.1371/journal.pone.0150171}
  {\path{doi:10.1371/journal.pone.0150171}}.

\bibitem{BruK19}
S.~L. Brunton and J.~N. Kutz.
\newblock {\em Data-Driven Science and Engineering:{M}achine Learning,
  Dynamical Systems, and Control}.
\newblock Cambridge University Press, Cambridge, 2019.
\newblock \href {https://doi.org/10.1017/9781108380690}
  {\path{doi:10.1017/9781108380690}}.

\bibitem{BruPK16}
S.~L. Brunton, J.~L. Proctor, and J.~N. Kutz.
\newblock Discovering governing equations from data by sparse identification of
  nonlinear dynamical systems.
\newblock {\em Proc. Natl. Acad. Sci. U. S. A.}, 113(15):3932--3937, 2016.
\newblock \href {https://doi.org/10.1073/pnas.1517384113}
  {\path{doi:10.1073/pnas.1517384113}}.

\bibitem{ChaI07}
N.~Chafee and E.~N. Infante.
\newblock A bifurcation problem for a nonlinear partial differential equation
  of parabolic type.
\newblock {\em Appl. Anal.}, 4(1):17--37, 2007.
\newblock \href {https://doi.org/10.1080/00036817408839081}
  {\path{doi:10.1080/00036817408839081}}.

\bibitem{DrmP22}
Z.~Drma{\v{c}} and B.~Peherstorfer.
\newblock Learning low-dimensional dynamical-system models from noisy
  frequency-response data with {L}oewner rational interpolation.
\newblock In C.~Beattie, P.~Benner, M.~Embree, S.~Gugercin, and S.~Lefteriu,
  editors, {\em Realization and Model Reduction of Dynamical Systems}, pages
  39--57. Springer, Cham, 2022.
\newblock \href {https://doi.org/10.1007/978-3-030-95157-3_3}
  {\path{doi:10.1007/978-3-030-95157-3_3}}.

\bibitem{FliJ13}
M.~Fliess and C.~Join.
\newblock Model-free control.
\newblock {\em Int. J. Control}, 86(12):2228--2252, 2013.
\newblock \href {https://doi.org/10.1080/00207179.2013.810345}
  {\path{doi:10.1080/00207179.2013.810345}}.

\bibitem{FujVM18}
S.~Fujimoto, H.~Van~Hoof, and D.~Meger.
\newblock Addressing function approximation error in actor-critic methods.
\newblock {\em Proceedings of the 35th International Conference on Machine
  Learning, PMLR}, 80:1587--1596, 2018.
\newblock URL: \url{https://proceedings.mlr.press/v80/fujimoto18a.html}.

\bibitem{GeeWW23}
R.~Geelen, S.~Wright, and K.~Willcox.
\newblock Operator inference for non-intrusive model reduction with quadratic
  manifolds.
\newblock {\em Comput. Methods Appl. Mech. Eng.}, 403, Part~B:115717, 2023.
\newblock \href {https://doi.org/10.1016/j.cma.2022.115717}
  {\path{doi:10.1016/j.cma.2022.115717}}.

\bibitem{GhaW21}
O.~Ghattas and K.~Willcox.
\newblock Learning physics-based models from data: perspectives from inverse
  problems and model reduction.
\newblock {\em Acta Numer.}, 30:445--554, 2021.
\newblock \href {https://doi.org/10.1017/S0962492921000064}
  {\path{doi:10.1017/S0962492921000064}}.

\bibitem{GolV13}
G.~H. Golub and C.~F. Van~Loan.
\newblock {\em Matrix Computations}.
\newblock Johns Hopkins Studies in the Mathematical Sciences. Johns Hopkins
  University Press, Baltimore, fourth edition, 2013.

\bibitem{Goe19}
D.~G{\"o}rges.
\newblock Distributed adaptive linear quadratic control using distributed
  reinforcement learning.
\newblock {\em IFAC-Pap.}, 52(11):218--223, 2019.
\newblock 5th {IFAC} Conference on Intelligent Control and Automation Sciences
  {ICONS} 2019.
\newblock \href {https://doi.org/10.1016/j.ifacol.2019.09.144}
  {\path{doi:10.1016/j.ifacol.2019.09.144}}.

\bibitem{GosA18}
I.~V. Gosea and A.~C. Antoulas.
\newblock Data-driven model order reduction of quadratic-bilinear systems.
\newblock {\em Numer. Linear Algebra Appl.}, 25(6):e2200, 2018.
\newblock \href {https://doi.org/10.1002/nla.2200}
  {\path{doi:10.1002/nla.2200}}.

\bibitem{GraES21}
B.~Gravell, P.~M. Esfahani, and T.~Summers.
\newblock Learning optimal controllers for linear systems with multiplicative
  noise via policy gradient.
\newblock {\em {IEEE} Trans. Autom. Control}, 66(11):5283--5298, 2021.
\newblock \href {https://doi.org/10.1109/TAC.2020.3037046}
  {\path{doi:10.1109/TAC.2020.3037046}}.

\bibitem{HeiP81}
R.~F. Heinemann and A.~B. Poore.
\newblock Multiplicity, stability, and oscillatory dynamics of the tubular
  reactor.
\newblock {\em Chem. Eng. Sci.}, 36(9):1411--1419, 1981.
\newblock \href {https://doi.org/10.1016/0009-2509(81)80175-3}
  {\path{doi:10.1016/0009-2509(81)80175-3}}.

\bibitem{KadBCetal22}
T.~Kadeethum, F.~Ballarin, Y.~Choi, and D.~and O'Malley.
\newblock Non-intrusive reduced order modeling of natural convection in porous
  media using convolutional autoencoders: {C}omparison with linear subspace
  techniques.
\newblock {\em Adv. Water Resour.}, 160:104098, 2022.
\newblock \href {https://doi.org/10.1016/j.advwatres.2021.104098}
  {\path{doi:10.1016/j.advwatres.2021.104098}}.

\bibitem{KadOFetal21}
T.~Kadeethum, D.~O'Malley, J.~N. Fuhg, Y.~Choi, J.~Lee, H.~S. Viswanathan, and
  N.~Bouklas.
\newblock A framework for data-driven solution and parameter estimation of
  {PDE}s using conditional generative adversarial networks.
\newblock {\em Nat. Comput. Sci.}, 1(12):819--829, 2021.
\newblock \href {https://doi.org/10.1038/s43588-021-00171-3}
  {\path{doi:10.1038/s43588-021-00171-3}}.

\bibitem{KarMR12}
D.~C. Karnopp, D.~L. Margolis, and R.~C. Rosenberg.
\newblock {\em System Dynamics: {M}odeling, Simulation, and Control of
  Mechatronic Systems}.
\newblock Wiley, Hoboken, NJ, fifth edition, 2012.
\newblock \href {https://doi.org/10.1002/9781118152812}
  {\path{doi:10.1002/9781118152812}}.

\bibitem{KimCWetal22}
Y.~Kim, Y.~Choi, D.~Widemann, and T.~Zohdi.
\newblock A fast and accurate physics-informed neural network reduced order
  model with shallow masked autoencoder.
\newblock {\em J. Comput. Phys.}, 451:110841, 2022.
\newblock \href {https://doi.org/10.1016/j.jcp.2021.110841}
  {\path{doi:10.1016/j.jcp.2021.110841}}.

\bibitem{Kos22}
I.~Kostrikov.
\newblock {JAXRL}: {I}mplementations of reinforcement learning algorithms in
  {JAX} (version~0.0.7), September 2022.
\newblock Version commit SHA: \texttt{bc8030c}.
\newblock URL: \url{https://github.com/ikostrikov/jaxrl}.

\bibitem{KraPW24}
B.~Kramer, B.~Peherstorfer, and K.~E. Willcox.
\newblock Learning nonlinear reduced models from data with operator inference.
\newblock {\em Annu. Rev. Fluid Mech.}, 56:521--548, 2024.
\newblock \href {https://doi.org/10.1146/annurev-fluid-121021-025220}
  {\path{doi:10.1146/annurev-fluid-121021-025220}}.

\bibitem{KraW19}
B.~Kramer and K.~E. Willcox.
\newblock Nonlinear model order reduction via lifting transformations and
  proper orthogonal decomposition.
\newblock {\em {AIAA} J.}, 57(6):2297--2307, 2019.
\newblock \href {https://doi.org/10.2514/1.J057791}
  {\path{doi:10.2514/1.J057791}}.

\bibitem{Kra91}
M.~A. Kramer.
\newblock Nonlinear principal component analysis using autoassociative neural
  networks.
\newblock {\em {AIChE} J.}, 37(2):233--243, 1991.
\newblock \href {https://doi.org/10.1002/aic.690370209}
  {\path{doi:10.1002/aic.690370209}}.

\bibitem{KraODetal05}
B.~Krauskopf, H.~M. Osinga, E.~J. Doedel, M.~E. Henderson, J.~Guckenheimer,
  A.~Vladimirsky, M.~Dellnitz, and O.~Junge.
\newblock A survey of methods for computing (un)stable manifolds of vector
  fields.
\newblock {\em Int. J. Bifurc. Chaos Appl. Sci. Eng.}, 15(3):763--791, 2005.
\newblock \href {https://doi.org/10.1142/S0218127405012533}
  {\path{doi:10.1142/S0218127405012533}}.

\bibitem{KutBBetal16}
J.~N. Kutz, S.~L. Brunton, B.~W. Brunton, and J.~L. Proctor.
\newblock {\em Dynamic Mode Decomposition: {D}ata-Driven Modeling of Complex
  Systems}.
\newblock SIAM, Philadelphia, PA, 2016.
\newblock \href {https://doi.org/10.1137/1.9781611974508}
  {\path{doi:10.1137/1.9781611974508}}.

\bibitem{LiTZetal22}
Y.~Li, Y.~Tang, R.~Zhang, and N.~Li.
\newblock Distributed reinforcement learning for decentralized linear quadratic
  control: {A} derivative-free policy optimization approach.
\newblock {\em {IEEE} Trans. Autom. Control}, 67(12):6429--6444, 2022.
\newblock \href {https://doi.org/10.1109/TAC.2021.3128592}
  {\path{doi:10.1109/TAC.2021.3128592}}.

\bibitem{LilHPetal15}
T.~P. Lillicrap, J.~J. Hunt, A.~Pritzel, N.~Heess, T.~Erez, Y.~Tassa,
  D.~Silver, and D.~Wierstra.
\newblock Continuous control with deep reinforcement learning.
\newblock e-print 1509.02971, arXiv, 2015.
\newblock Machine Learning (cs.LG).
\newblock \href {https://doi.org/10.48550/arXiv.1509.02971}
  {\path{doi:10.48550/arXiv.1509.02971}}.

\bibitem{Loc01}
A.~Locatelli.
\newblock {\em Optimal Control: {A}n Introduction}.
\newblock Birkh{\"a}user, Basel, 2001.

\bibitem{MalPBetal19}
D.~Malik, A.~Pananjady, K.~Bhatia, K.~Khamaru, P.~Bartlett, and M.~Wainwright.
\newblock Derivative-free methods for policy optimization: {G}uarantees for
  linear quadratic systems.
\newblock In K.~Chaudhuri and M.~Sugiyama, editors, {\em Proceedings of the
  Twenty-Second International Conference on Artificial Intelligence and
  Statistics, PMLR}, pages 2916--2925, 2019.
\newblock URL: \url{https://proceedings.mlr.press/v89/malik19a.html}.

\bibitem{MayA07}
A.~J. Mayo and A.~C. Antoulas.
\newblock A framework for the solution of the generalized realization problem.
\newblock {\em Linear Algebra Appl.}, 425(2--3):634--662, 2007.
\newblock Special issue in honor of P.~A. Fuhrmann, Edited by A.~C. Antoulas,
  U. Helmke, J. Rosenthal, V. Vinnikov, and E. Zerz.
\newblock \href {https://doi.org/10.1016/j.laa.2007.03.008}
  {\path{doi:10.1016/j.laa.2007.03.008}}.

\bibitem{Mez05}
I.~Mezi{\'c}.
\newblock Spectral properties of dynamical systems, model reduction and
  decompositions.
\newblock {\em Nonlinear Dyn.}, 41(1--3):309--325, 2005.
\newblock \href {https://doi.org/10.1007/s11071-005-2824-x}
  {\path{doi:10.1007/s11071-005-2824-x}}.

\bibitem{MitFD19}
S.~K. Mitusch, S.~W. Funke, and J.~S. Dokken.
\newblock dolfin-adjoint 2018.1: automated adjoints for {FEniCS} and
  {F}iredrake.
\newblock {\em J. Open Source Softw.}, 4(38):1292, 2019.
\newblock \href {https://doi.org/10.21105/joss.01292}
  {\path{doi:10.21105/joss.01292}}.

\bibitem{MohZSetal22}
H.~Mohammadi, A.~Zare, M.~Soltanolkotabi, and M.~R. Jovanovi{\'c}.
\newblock Convergence and sample complexity of gradient methods for the
  model-free linear--quadratic regulator problem.
\newblock {\em {IEEE} Trans. Autom. Control}, 67(5):2435--2450, 2022.
\newblock \href {https://doi.org/10.1109/TAC.2021.3087455}
  {\path{doi:10.1109/TAC.2021.3087455}}.

\bibitem{Moo81}
B.~C. Moore.
\newblock Principal component analysis in linear systems: controllability,
  observability, and model reduction.
\newblock {\em {IEEE} Trans. Autom. Control}, AC--26(1):17--32, 1981.
\newblock \href {https://doi.org/10.1109/TAC.1981.1102568}
  {\path{doi:10.1109/TAC.1981.1102568}}.

\bibitem{MusG91}
D.~Mustafa and K.~Glover.
\newblock Controller reduction by {$\mathcal{H}_{\infty}$}-balanced truncation.
\newblock {\em {IEEE} Trans. Autom. Control}, 36(6):668--682, 1991.
\newblock \href {https://doi.org/10.1109/9.86941} {\path{doi:10.1109/9.86941}}.

\bibitem{NijV16}
H.~Nijmeijer and A.~Van~der Schaft.
\newblock {\em Nonlinear Dynamical Control Systems}.
\newblock Springer, New York, NY, fourth edition, 2016.
\newblock \href {https://doi.org/10.1007/978-1-4757-2101-0}
  {\path{doi:10.1007/978-1-4757-2101-0}}.

\bibitem{PasGMetal19}
A.~Paszke, S.~Gross, F.~Massa, A.~Lerer, J.~Bradbury, G.~Chanan, T.~Killeen,
  Z.~Lin, N.~Gimelshein, L.~Antiga, A.~Desmaison, A.~Kopf, E.~Yang, Z.~DeVito,
  M.~Raison, A.~Tejani, S.~Chilamkurthy, B.~Steiner, L.~Fang, J.~Bai, and
  S.~Chintala.
\newblock {PyTorch}: {A}n imperative style, high-performance deep learning
  library.
\newblock In H.~Wallach, H.~Larochelle, A.~Beygelzimer, F.~d'Alch{\'e} Buc,
  E.~Fox, and R.~Garnett, editors, {\em Advances in Neural Information
  Processing Systems}, volume~32, pages 8026--8037, 2019.

\bibitem{Pea01}
K.~Pearson.
\newblock On lines and planes of closest fit to systems of points in space.
\newblock {\em Philos. Mag.}, 2(11):559--572, 1901.
\newblock \href {https://doi.org/10.1080/14786440109462720}
  {\path{doi:10.1080/14786440109462720}}.

\bibitem{Peh20}
B.~Peherstorfer.
\newblock Sampling low-dimensional {M}arkovian dynamics for preasymptotically
  recovering reduced models from data with operator inference.
\newblock {\em {SIAM} J. Sci. Comput.}, 42(5):A3489--A3515, 2020.
\newblock \href {https://doi.org/10.1137/19M1292448}
  {\path{doi:10.1137/19M1292448}}.

\bibitem{PehW16}
B.~Peherstorfer and K.~Willcox.
\newblock Data-driven operator inference for nonintrusive projection-based
  model reduction.
\newblock {\em Comput. Methods Appl. Mech. Eng.}, 306:196--215, 2016.
\newblock \href {https://doi.org/10.1016/j.cma.2016.03.025}
  {\path{doi:10.1016/j.cma.2016.03.025}}.

\bibitem{PehWG18}
B.~Peherstorfer, K.~Willcox, and M.~Gunzburger.
\newblock Survey of multifidelity methods in uncertainty propagation,
  inference, and optimization.
\newblock {\em {SIAM} Rev.}, 60(3):550--591, 2018.
\newblock \href {https://doi.org/10.1137/16M1082469}
  {\path{doi:10.1137/16M1082469}}.

\bibitem{QiaKPetal20}
E.~Qian, B.~Kramer, B.~Peherstorfer, and K.~Willcox.
\newblock {L}ift {\&} {L}earn: {P}hysics-informed machine learning for
  large-scale nonlinear dynamical systems.
\newblock {\em Phys. D: Nonlinear Phenom.}, 406:132401, 2020.
\newblock \href {https://doi.org/10.1016/j.physd.2020.132401}
  {\path{doi:10.1016/j.physd.2020.132401}}.

\bibitem{Rec19}
B.~Recht.
\newblock A tour of reinforcement learning: {T}he view from continuous control.
\newblock {\em Annu. Rev. Control Robot. Auton. Syst.}, 2:253--279, 2019.
\newblock \href {https://doi.org/10.1146/annurev-control-053018-023825}
  {\path{doi:10.1146/annurev-control-053018-023825}}.

\bibitem{RowMBetal09}
C.~W. Rowley, I.~Mezi{\'c}, S.~Bagheri, P.~Schlatter, and D.~S. Henningson.
\newblock Spectral analysis of nonlinear flows.
\newblock {\em J. Fluid Mech.}, 641:115--127, 2009.
\newblock \href {https://doi.org/10.1017/S0022112009992059}
  {\path{doi:10.1017/S0022112009992059}}.

\bibitem{SawKP23}
N.~Sawant, B.~Kramer, and B.~Peherstorfer.
\newblock Physics-informed regularization and structure preservation for
  learning stable reduced models from data with operator inference.
\newblock {\em Comput. Methods Appl. Mech. Eng.}, 404:115836, 2023.
\newblock \href {https://doi.org/10.1016/j.cma.2022.115836}
  {\path{doi:10.1016/j.cma.2022.115836}}.

\bibitem{SchTW18}
H.~Schaeffer, G.~Tran, and R.~Ward.
\newblock Extracting sparse high-dimensional dynamics from limited data.
\newblock {\em {SIAM} J. Appl. Math.}, 78(6):3279--3295, 2018.
\newblock \href {https://doi.org/10.1137/18M116798X}
  {\path{doi:10.1137/18M116798X}}.

\bibitem{Sch10}
P.~J. Schmid.
\newblock Dynamic mode decomposition of numerical and experimental data.
\newblock {\em J. Fluid Mech.}, 656:5--28, 2010.
\newblock \href {https://doi.org/10.1017/S0022112010001217}
  {\path{doi:10.1017/S0022112010001217}}.

\bibitem{SchU16}
P.~Schulze and B.~Unger.
\newblock Data-driven interpolation of dynamical systems with delay.
\newblock {\em Syst. Control Lett.}, 97:125--131, 2016.
\newblock \href {https://doi.org/10.1016/j.sysconle.2016.09.007}
  {\path{doi:10.1016/j.sysconle.2016.09.007}}.

\bibitem{SchUBetal18}
P.~Schulze, B.~Unger, C.~Beattie, and S.~Gugercin.
\newblock Data-driven structured realization.
\newblock {\em Linear Algebra Appl.}, 537:250--286, 2018.
\newblock \href {https://doi.org/10.1016/j.laa.2017.09.030}
  {\path{doi:10.1016/j.laa.2017.09.030}}.

\bibitem{SchP24}
P.~Schwerdtner and B.~Peherstorfer.
\newblock Greedy construction of quadratic manifolds for nonlinear
  dimensionality reduction and nonlinear model reduction.
\newblock e-print 2403.06732, arXiv, 2024.
\newblock Numerical Analysis (math.NA).
\newblock \href {https://doi.org/10.48550/arXiv.2403.06732}
  {\path{doi:10.48550/arXiv.2403.06732}}.

\bibitem{ShaWK22}
H.~Sharma, Z.~Wang, and B.~Kramer.
\newblock {H}amiltonian operator inference: {P}hysics-preserving learning of
  reduced-order models for canonical {H}amiltonian systems.
\newblock {\em Phys. D: Nonlinear Phenom.}, 431:133122, 2022.
\newblock \href {https://doi.org/10.1016/j.physd.2021.133122}
  {\path{doi:10.1016/j.physd.2021.133122}}.

\bibitem{Sil18}
E.~Silverman.
\newblock {\em Methodological Investigations in Agent-Based Modelling: {W}ith
  Applications for the Social Sciences}, volume~13 of {\em Methodos Series}.
\newblock Springer, Cham, 2018.
\newblock \href {https://doi.org/10.1007/978-3-319-72408-9}
  {\path{doi:10.1007/978-3-319-72408-9}}.

\bibitem{Ste01}
G.~W. Stewart.
\newblock A {K}rylov--{S}chur algorithm for large eigenproblems.
\newblock {\em {SIAM} J. Matrix Anal. Appl.}, 23(3):601--614, 2001.
\newblock \href {https://doi.org/10.1137/S0895479800371529}
  {\path{doi:10.1137/S0895479800371529}}.

\bibitem{SutB18}
R.~S. Sutton and A.~G. Barto.
\newblock {\em Reinforcement Learning: {A}n Introduction}.
\newblock MIT Press, Cambridge, MA, second edition, 2018.
\newblock URL: \url{http://incompleteideas.net/book/the-book-2nd.html}.

\bibitem{Tod67}
M.~Toda.
\newblock Vibration of a chain with nonlinear interaction.
\newblock {\em J. Phys. Soc. Jpn.}, 22(2):431--436, 1967.
\newblock \href {https://doi.org/10.1143/JPSJ.22.431}
  {\path{doi:10.1143/JPSJ.22.431}}.

\bibitem{TuRLetal14}
J.~H. Tu, C.~W. Rowley, D.~M. Luchtenburg, S.~L. Brunton, and J.~N. Kutz.
\newblock On dynamic mode decomposition: {T}heory and applications.
\newblock {\em J. Comput. Dyn.}, 1(2):391--421, 2014.
\newblock \href {https://doi.org/10.3934/jcd.2014.1.391}
  {\path{doi:10.3934/jcd.2014.1.391}}.

\bibitem{VlaAUetal22}
P.~R. Vlachas, G.~Arampatzis, C.~Uhler, and P.~Koumoutsakos.
\newblock Multiscale simulations of complex systems by learning their effective
  dynamics.
\newblock {\em Nat. Mach. Intell.}, 4(4):359--366, 2022.
\newblock \href {https://doi.org/10.1038/s42256-022-00464-w}
  {\path{doi:10.1038/s42256-022-00464-w}}.

\bibitem{Wer21}
S.~W.~R. Werner.
\newblock {\em Structure-Preserving Model Reduction for Mechanical Systems}.
\newblock {D}issertation, Otto-von-Guericke-Universit{\"a}t, Magdeburg,
  Germany, 2021.
\newblock \href {https://doi.org/10.25673/38617} {\path{doi:10.25673/38617}}.

\bibitem{supWer24}
S.~W.~R. Werner.
\newblock Code, data and results for numerical experiments in ``{S}ystem
  stabilization with policy optimization on unstable latent manifolds''
  (version 1.0), April 2024.
\newblock \href {https://doi.org/10.5281/zenodo.7897240}
  {\path{doi:10.5281/zenodo.7897240}}.

\bibitem{WerOP23}
S.~W.~R. Werner, M.~L. Overton, and B.~Peherstorfer.
\newblock Multifidelity robust controller design with gradient sampling.
\newblock {\em {SIAM} J. Sci. Comput.}, 45(2):A933--A957, 2023.
\newblock \href {https://doi.org/10.1137/22M1500137}
  {\path{doi:10.1137/22M1500137}}.

\bibitem{WerP23a}
S.~W~.R. Werner and B.~Peherstorfer.
\newblock Context-aware controller inference for stabilizing dynamical systems
  from scarce data.
\newblock {\em Proc. R. Soc. A: Math. Phys. Eng. Sci.}, 479(2270):20220506,
  2023.
\newblock \href {https://doi.org/10.1098/rspa.2022.0506}
  {\path{doi:10.1098/rspa.2022.0506}}.

\bibitem{WerP24}
S.~W~.R. Werner and B.~Peherstorfer.
\newblock On the sample complexity of stabilizing linear dynamical systems from
  data.
\newblock {\em Found. Comput. Math.}, 24(3):955--987, 2024.
\newblock \href {https://doi.org/10.1007/s10208-023-09605-y}
  {\path{doi:10.1007/s10208-023-09605-y}}.

\bibitem{WilKR15}
M.~O. Williams, I.~G. Kevrekidis, and C.~W. Rowley.
\newblock A data-driven approximation of the {K}oopman operator: {E}xtending
  dynamic mode decomposition.
\newblock {\em J. Nonlinear Sci.}, 25(6):1307--1346, 2015.
\newblock \href {https://doi.org/10.1007/s00332-015-9258-5}
  {\path{doi:10.1007/s00332-015-9258-5}}.

\bibitem{ZieN42}
J.~Ziegler and N.~Nichols.
\newblock Optimum settings for automatic controllers.
\newblock {\em Trans. ASME}, 64:759--768, 1942.

\bibitem{ZieDG19}
A.~Ziessler, M.~Dellnitz, and R.~Gerlach.
\newblock The numerical computation of unstable manifolds for infinite
  dimensional dynamical systems by embedding techniques.
\newblock {\em {SIAM} J. Appl. Dyn. Syst.}, 18(3):1265--1292, 2019.
\newblock \href {https://doi.org/10.1137/18M1204395}
  {\path{doi:10.1137/18M1204395}}.

\end{thebibliography}

\end{document}